\documentclass[12pt]{amsart} 

\usepackage{amsthm,amssymb}
\usepackage{fullpage}
\usepackage[utf8]{inputenc}
\usepackage{graphicx} 
\usepackage{subcaption}
\usepackage{amssymb}
\usepackage{xcolor}
\usepackage{float}
\usepackage{array}
\usepackage{graphicx, array, blindtext}
\input xy
\xyoption{all}

\newcommand{\stw}{\mathbb{S}^2}
\newcommand{\sth}{\mathbb{S}^3}

\newcommand{\rth}{\mathbb{R}^3}

\newcommand{\embed}{L}

\usepackage[usestackEOL]{stackengine}

\newtheorem{proposition}{Proposition}
\newtheorem{theorem}{Theorem}
\newtheorem{lemma}{Lemma}
\newtheorem{corollary}{Corollary}
\newtheorem{question}{Question}
\newtheorem{problem}{Problem}
\newtheorem{remark}{Remark}

\begin{document}
\title[Self-dual maps III: projective links]{Self-dual maps III: projective links} 

\thanks{$^1$ Partially supported by CONACyT 166306 and PAPIIT-UNAM IN112614}
\thanks{$^2$ Partially supported by grant PICS07848 and INSMI-CNRS}
\author[Luis Montejano]{Luis Montejano$^1$}
\address{Instituto de Matem\'aticas, Universidad Nacional A. de M\'exico at Quer\'etaro
Quer\'etaro, M\'exico, CP. 07360}
\email{luis@im.unam.mx}
\author[Jorge L. Ram\'irez Alfons\'in]{Jorge L. Ram\'irez Alfons\'in$^2$}
\address{
UMI2924 - J.-C. Yoccoz, CNRS-IMPA, Brazil and IMAG, Univ.\ Montpellier, CNRS, France }
\email{jorge.ramirez-alfonsin@umontpellier.fr}
\author[Ivan Rasskin]{Ivan Rasskin}
\address{IMAG, Univ.\ Montpellier, CNRS, Montpellier, France}
\email{ivan.rasskin@umontpellier.fr}



\begin{abstract} In this paper, we present necessary and sufficient combinatorial conditions for a link to be {\em projective}, that is, a link in $\mathbb{R}\mathbb{P}^3$. This characterization is closely related to the notions of {\em antipodally self-dual} and {\em antipodally symmetric maps}. We also discuss the notion of {\em symmetric cycle}, an interesting issue arising in projective links leading us to an easy condition to prevent a projective link to be {\em alternating}.

\end{abstract}

\maketitle

\section{Introduction}

This paper is a continuation of the work \cite{MRAR1} where the notions of antipodally self-dual and antipodally symmetric maps were studied and the work \cite{MRAR2} where these notions were applied to investigate questions concerning {\em symmetry} and {\em amphicharility} of links. It turns out that the above notions fit also nicely in order to understand better {\em projective} links, that is, links in $\mathbb{R}\mathbb{P}^3$. 
\smallskip

Projective links have been studied in different contexts : in connection with the twisted Alexander polynomial \cite{HL} and also in relation with unknotting issues \cite{Mro} and possibly with 3-manifolds \cite[Chapter IX]{PS}. In \cite{Drobo}, Drobotukhina presented the analogue of the Jones polynomial for projective links and then used it in \cite{Drobo1} to give a classification of projective links with at most six crossings. 
\smallskip

In this paper, we present necessary and sufficient combinatorial conditions for a link to be projective. The latter was done by using a characterization of projective links in terms of some special embeddings in 3-space invariant under {\em negative inversions}. 
\smallskip

In the next section we overview some basic notions of knots, projective links and maps needed for the rest of the paper. In Section \ref{sec;newapproach}, we present our new approach and the above mentioned characterizations. In Section \ref{sec:conclud}, we discuss the notion of {\em symmetric cycle}, an interesting issue arising in projective links leading us to an easy condition to prevent a projective link to be {\em alternating}.

\section{Knots and maps : preliminaries}\label{sec;map} 

\subsection{Knots background} We refer the reader to \cite{Adam} of \cite{Liv} for standard background on knot theory. 
\smallskip

A {\em link} with $k$ components consists of $k$ disjoint simple closed curves ($\mathbb{S}^1$) in $\mathbb{R}^3$. A {\em knot} $K$ is a link with one component.  A {\em link diagram} $D(L)$ of a link $L$ is a regular projection of $L$ into $\mathbb{R}^2$ in such a way that the projection of each component is smooth and at most two curves intersect at any point. At each crossing point of the link diagram the curve which goes over the other is specified. A {\em shadow} of a link diagram $D$ is a 4-regular graph if the over/under passes of $D$ are ignored. Since the shadow is Eulerian (4-regular) then its faces can be 2-colored, say with colors black and white. We thus have that each vertex is incident to 4 faces alternatively colored around the vertex, see Figure \ref{fig2}.

\begin{figure}[H]
\centering
\includegraphics[width=.7\linewidth]{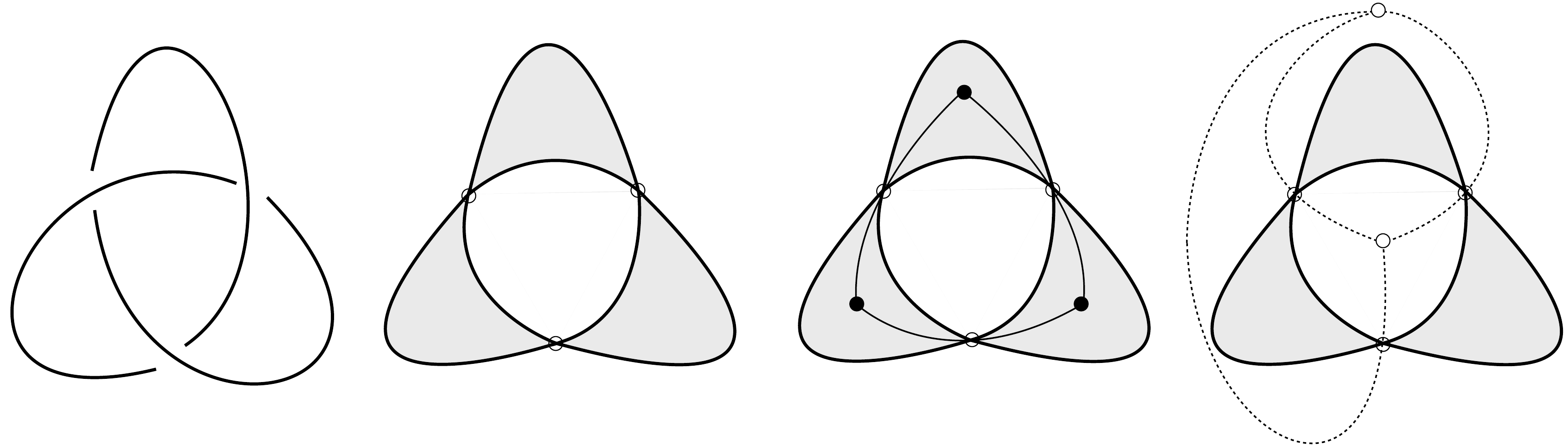}
\caption{(From left to right) A diagram of the Trefoil, its shadow with a 2-colored faces (vertices on white crossed circles), corresponding Black graph (bold edges and black circles) and White graph (dotted edges and white circles).}
\label{fig2}
\end{figure}

Given such a coloring, we can define two graphs, one on the faces of each color. Let $B_D$ denote the graph with black faces as its vertices and two vertices are joined if the corresponding faces share a vertex ($B_D$ is called the {\em checkerboard graph} of $D$). We define the graph $W_D$ on the white faces of the shadow analogously. We notice that $med(W_D)$ is exactly the shadow of $D$ and notice that $med(B_D)$ and $med(W_D)$ are the same.

\smallskip

An \textit{edge-signed} planar graph, denoted by $(G,S_E)$, is a planar graph $G$ equipped with a signature on its edges $S_E:E\rightarrow \{+,-\}$. We will denote by $-S_E$ the signature of $G$ satisfying $-S_E(e)=-(S_E(e))$ for every $e\in E$. We write $S_E^+$ (resp. $S_E^-$) when all the signs of $S_E$ are $+$ (resp. $-$). Given a crossing of the link diagram we sign {\em positive} or {\em negative} according to the {\em left-over-right} and {\em right-over-left} rules from point of view of black around the crossing, see Figure \ref{fig3} (Left). The latter induce an opposite signing on each crossing by the same rules but now from point of view of white around the crossing, see Figure \ref{fig3} (Right).

\begin{figure}[H]
\centering
\includegraphics[width=.6\linewidth]{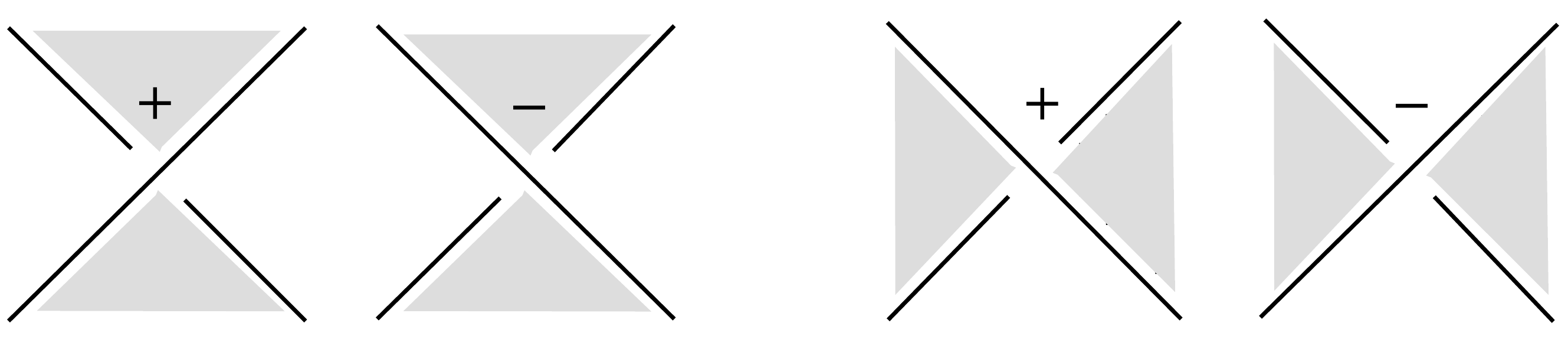}
\caption{(Left) Left-over-right rule from black point of view. (Right) Right-over-left rule from white point of view.}
\label{fig3}
\end{figure}


If the crossing is positive, relative to the black faces, then the corresponding edge is declared to be positive in $B_D$ and negative in $W_D$. Therefore, in this fashion, a link diagram $D$ determine a dual pair of signed planar graphs $(B_D,S_E)$ and $(W_D,-S_E)$ where the signs on edges are swapped on moving to the dual.  

\begin{remark}\label{ram;blackwhite} A link diagram can be uniquely recovered from either $(B_D,S_E)$ or $(W_D,-S_E)$.
\end{remark}

We thus have that given an edge-signed planar graph $(G,S_E)$, we can associate to it (in a canonical way) a link diagram $D(L)$ such that $(B_D,S_E)$ (and $(W_D,-S_E)$) gives $(G,S_E)$. The unsigned graph $G$ is called the {\em Tait graph} of the link $L$ with diagram $D$. The construction is easy, we just consider the $med(H)$ with signatures on the vertices (induced by the edge-signature $S_E$ of $G$). The desired diagram, denoted by $D(G,S_E)$, is obtained by determining the under/over pass at each crossing according to Left-over-right (or Right-over-left) rule associated to the sign of the corresponding edge of $G$, see Figure \ref{fig4}.

\begin{figure}[H]
\centering
\includegraphics[width=.7\linewidth]{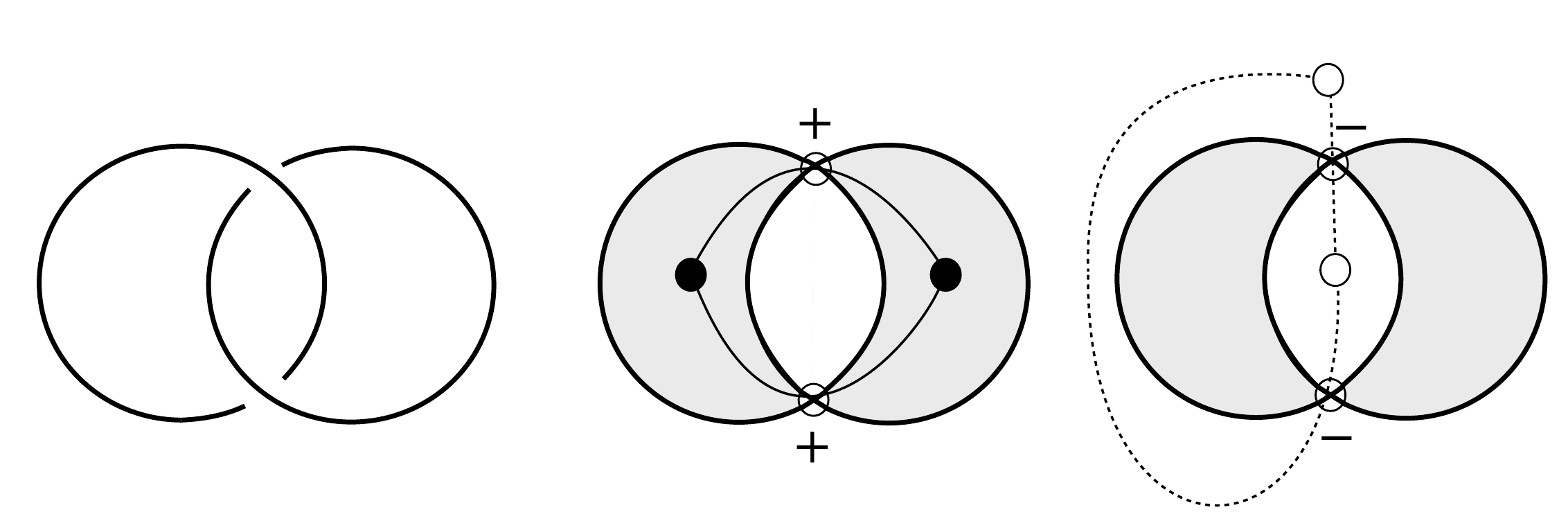}
\caption{(From left to right) diagram $D$ of the Hopf link, denoted by $2_1$, signed Black graph $(B_D,S_E)$ and signed White graph $(W_D,-S_E)$.}
\label{fig4}
\end{figure}

\subsection{Diagrams for links in $\mathbb{R}\mathbb{P}^3$}\label{subsect:pro}
A {\em projective} $n$-link is the image of a smooth embedding of $n$ disjoint copies of $\mathbb{S}^1$ in $\mathbb{R}\mathbb{P}^3$.
\smallskip

The 3-dimension real projective space $\mathbb{R}\mathbb{P}^3$ can be defined as the sphere $\mathbb{S}^3$ with identified opposite points. Since $\mathbb{S}^3$ consists of two half-spheres then we can restrict ourselves to the upper hemisphere of  $\mathbb{S}^3$ and merely identify antipodal points on the bounding equator. Now, since each half-sphere is homeomorphic to the ball $\mathbb{B}^3$ then have that $\mathbb{R}\mathbb{P}^3$ can be obtained by identifying diametrically opposite points of the boundary $\mathbb{S}^2=\partial \mathbb{B}^3$. 
\smallskip
 
Consequently, any link $L$ in $\mathbb{R}\mathbb{P}^3$ can be defined as a set of closed curves and arcs in $\mathbb{B}^3$ such that the set of endpoints of arcs lies in $\partial\mathbb{B}^3$.  
We say that the link $L$  in $\mathbb{R}\mathbb{P}^3$ is {\em lifted} to $L'$ in $\mathbb{B}^3$. Up to isotopy, we can assume that the images of the poles of $\mathbb{B}^3$ do not belong to $L'$.  

A projective link $L$ can be represented by {\em projective diagrams} that differ from usuals link diagrams in $\mathbb{R}^3$ in that they are given not on the plane but in a closed 2-disc, and the endpoints of arcs at the boundary of the 2-disc are divided into pairs of diametrically opposite points.  

More precisely,  let $\pi: L'\rightarrow \delta$ be the projection of $L'$ to the equatorial disc $\Delta\subset \mathbb{B}^3$ defined by
$$\pi(x)=C_x\cap \Delta$$
where $C_x\subset \mathbb{B}^3$ is the semicircle passing through $x$ and the poles of the ball. Up to a small isotopy, we can assume that $L$ satisfies the following general position properties:
\smallskip

a) the image $\pi(L')$ does not contain any cusp, tangency point or triple point,
\smallskip

b) $L'$ is a smooth submanifold of $\mathbb{B}^3 $ intersecting transversally the boundary sphere $\partial\mathbb{B}^3$
\smallskip

c) there do not exist 2 points $L'\cap \partial\mathbb{B}^3$ projecting to the same point under $\pi$.
\smallskip


We thus have that a link in $\mathbb{R}\mathbb{P}^3$ give rise a projective diagram by projecting to the equatorial disc $\mathbb{B}^2$. Conversely, given a projective diagram regarding $\mathbb{B}^2$ as the equatorial disc of such a representation of $\mathbb{R}\mathbb{P}^3$. At each crossing, the upper arc is pull up (and the lower one is pull down) in order to obtain a nonintersecting curve lying inside $\mathbb{B}^3$. The identification of antipodal points lying on the boundary sphere $\partial\mathbb{B}^3$ gives rise to a link in $\mathbb{R}\mathbb{P}^3$.


\begin{figure}[!htp]
\centering
\includegraphics[width=0.5\linewidth]{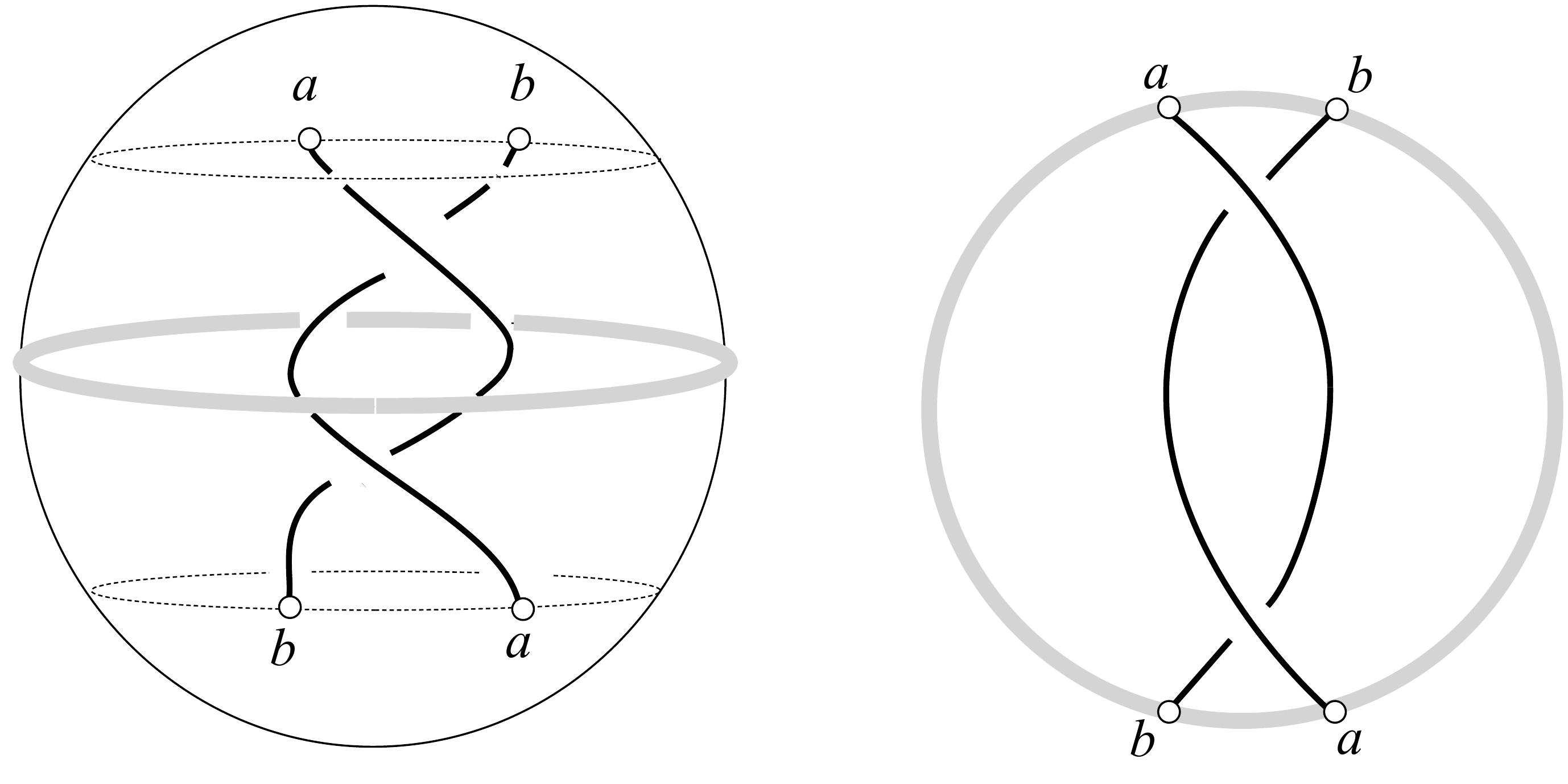}
\caption{Nonintersecting curves  inside $\mathbb{B}^3$ (induced by a link in $\mathbb{R}\mathbb{P}^3$) and its projective diagram in $\mathbb{R}\mathbb{P}^2$.}
\label{fig21}
\end{figure}

In Table \ref{tab:proj} we present the firs 14 nontrivial projective links among the 111 projective links with at most 6 crossings appearing in \cite[Table of links in $\mathbb{R}\mathbb{P}^3$, page 102]{Drobo1}.
\medskip

\begin{table}[ht]
\centering
\caption{First nontrivial projective links.}
\label{tab:proj}
\begin{tabular}[t]{cccc}
  \includegraphics[width=0.15\linewidth]{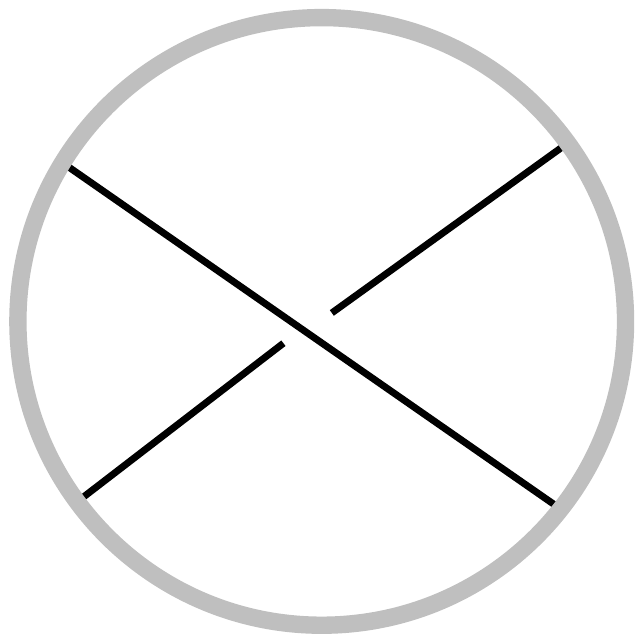} & 
 \includegraphics[width=0.15\linewidth]{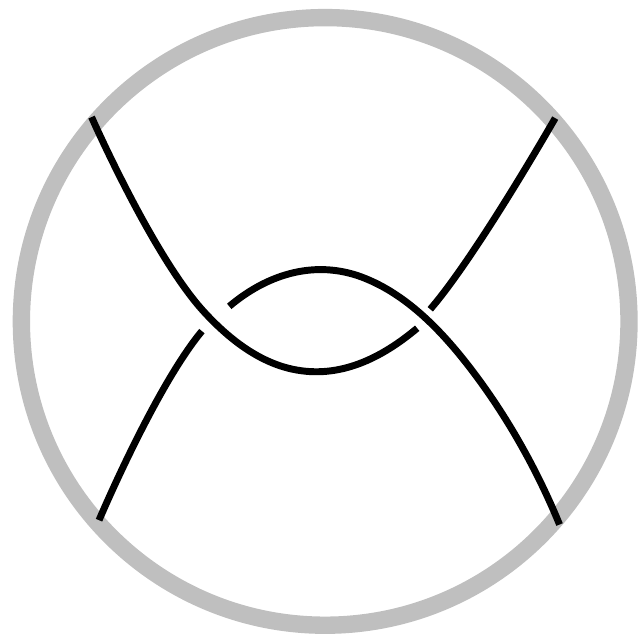} & 
  \includegraphics[width=0.15\linewidth]{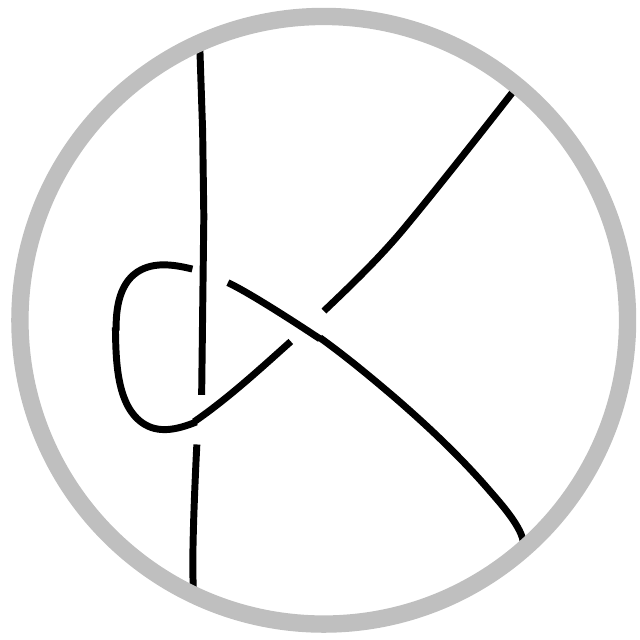} & 
  \includegraphics[width=0.15\linewidth]{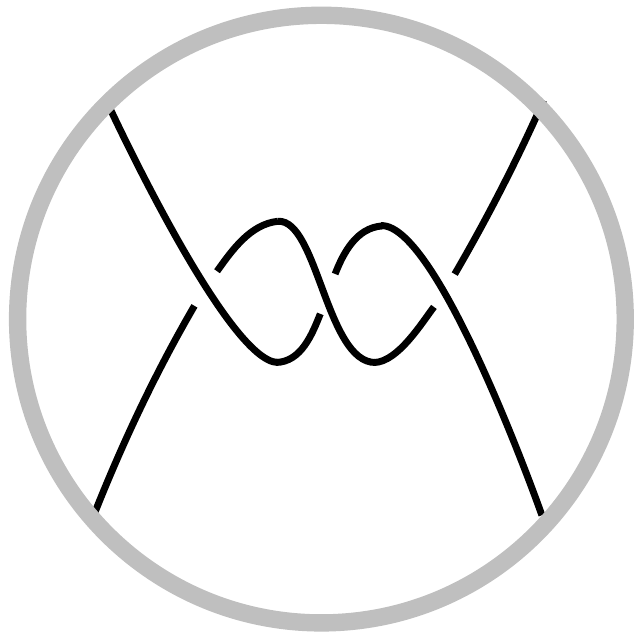} \\
  \includegraphics[width=0.15\linewidth]{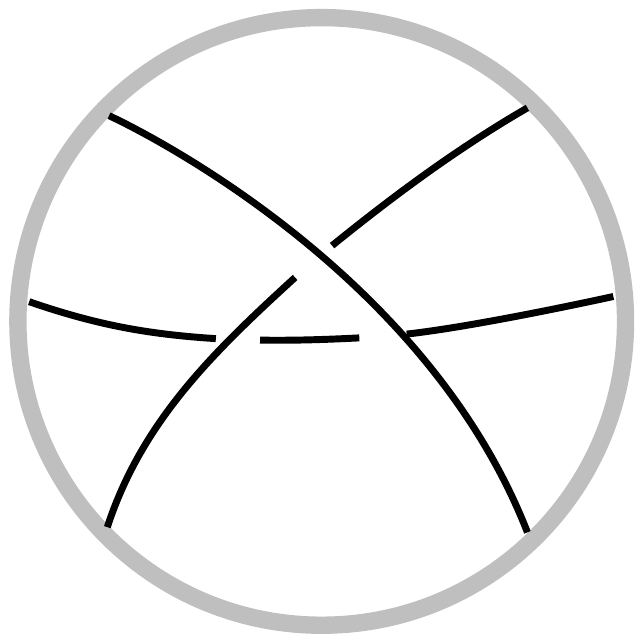} & 
 \includegraphics[width=0.15\linewidth]{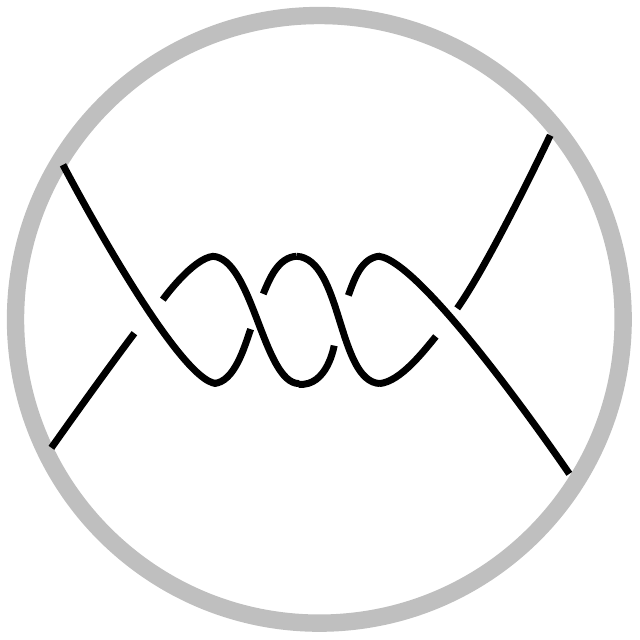} & 
  \includegraphics[width=0.15\linewidth]{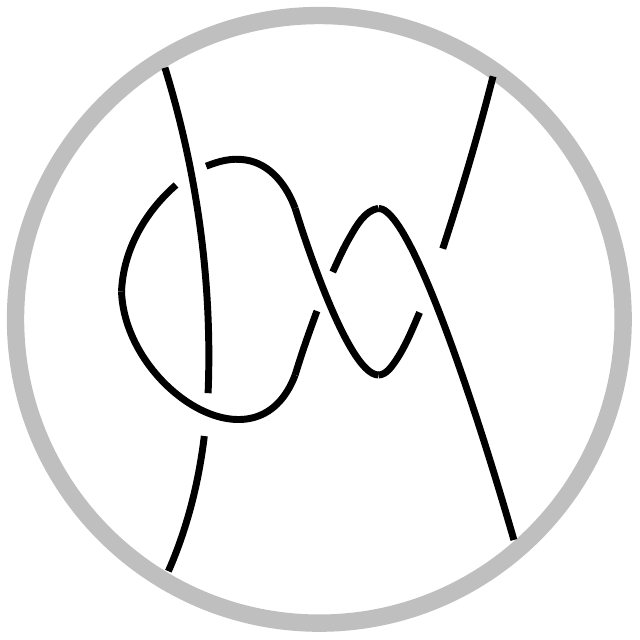} & 
  \includegraphics[width=0.15\linewidth]{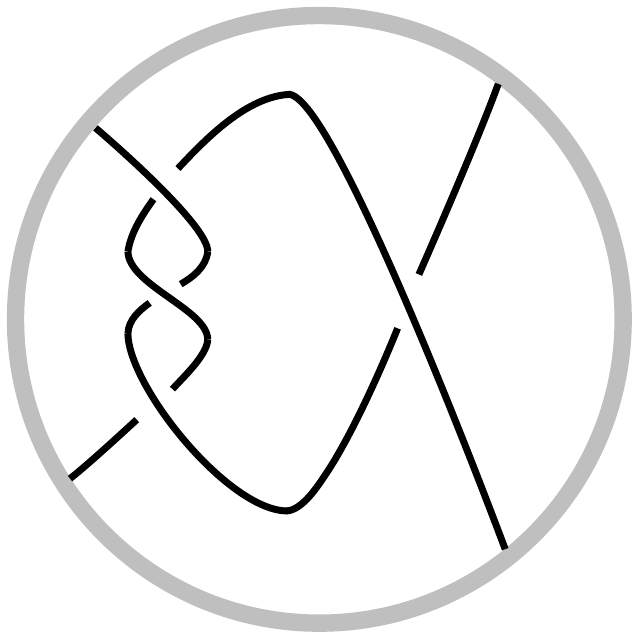} \\
   \includegraphics[width=0.15\linewidth]{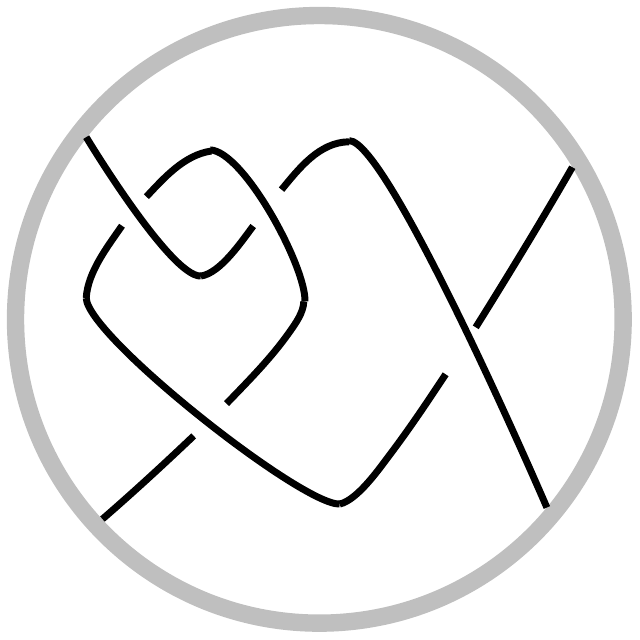} & 
 \includegraphics[width=0.15\linewidth]{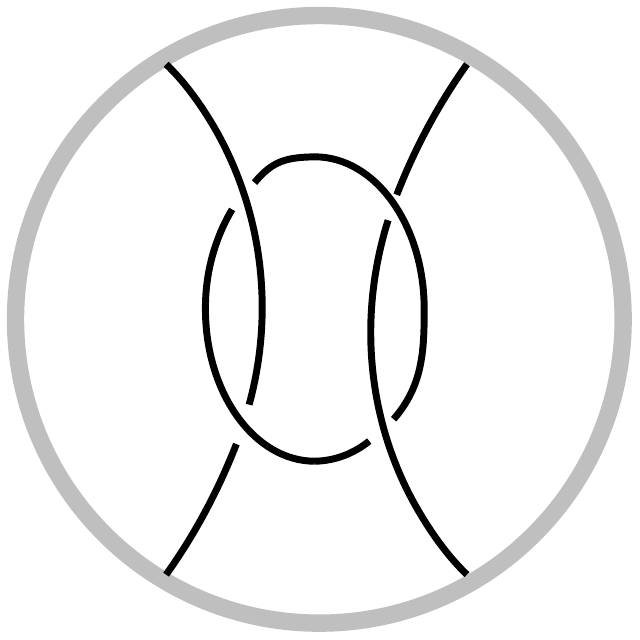} & 
  \includegraphics[width=0.15\linewidth]{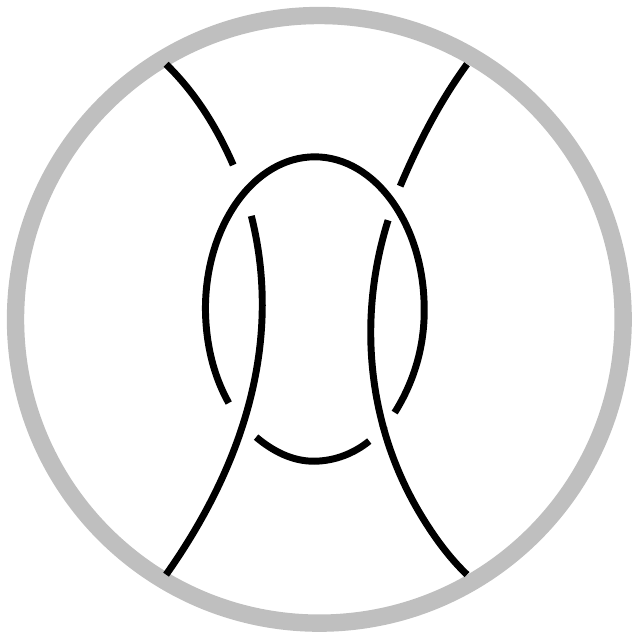} & 
  \includegraphics[width=0.15\linewidth]{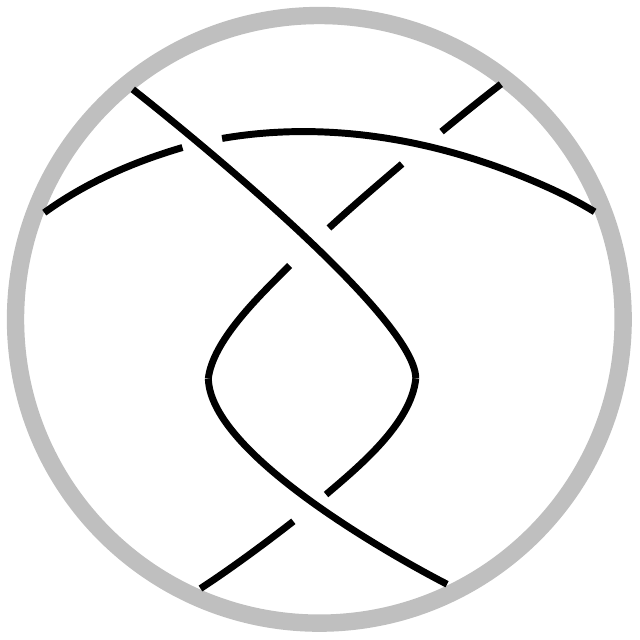} \\
   \includegraphics[width=0.15\linewidth]{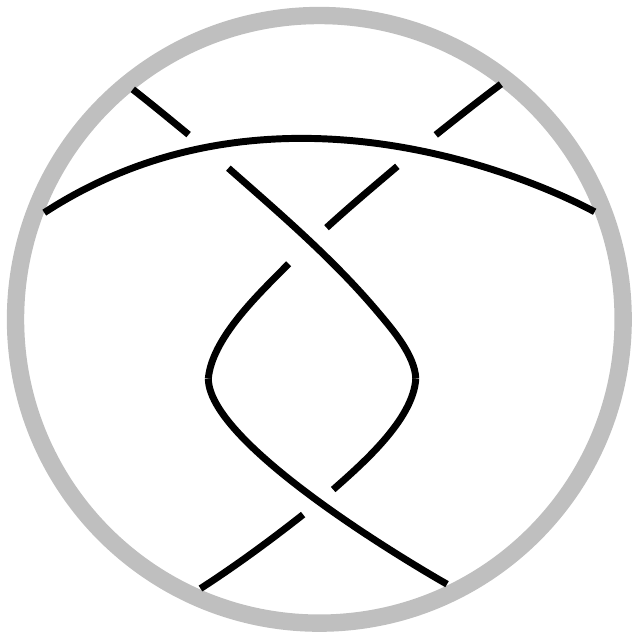} & 
 \includegraphics[width=0.15\linewidth]{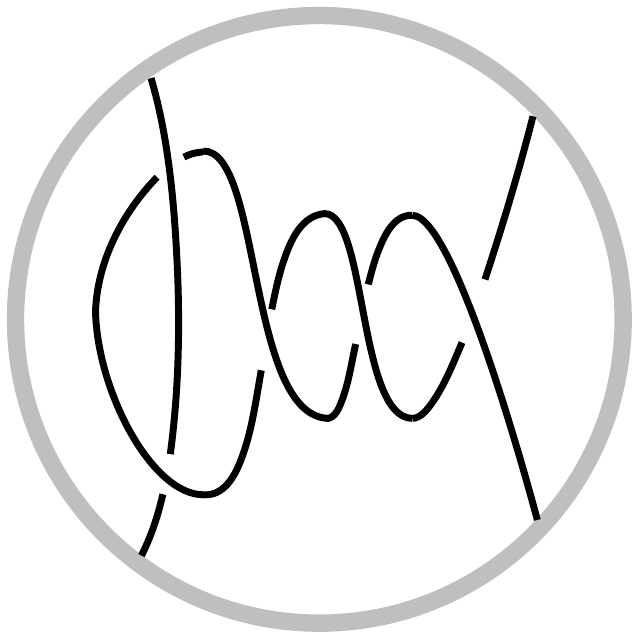} & &  \\
 \end{tabular}
\end{table}

We notice that any projective diagram in $\mathbb{R}\mathbb{P}^2$ arises a projective link. However, it may happens that two such diagrams lead to the same projective link. 

\begin{problem} Let $D_1$ and $D_2$ be two  diagram in $\mathbb{R}\mathbb{P}^2$. Determine if $D_1$ and $D_2$ arise the same projective link.
\end{problem}

This seems a very tough problem (maybe as hard as to determine if two links are isotopic). We discuss further this in Subsection \ref{subsec:diagram}.

\subsection{Maps background} 
A {\em map} of $G=(V,E,F)$ is the image of an embedding of $G$ into $\mathbb{S}^2$ where the set of vertices are a collection of distinct points in $\mathbb{S}^2$ and the set of edges are a collection of Jordan curves joining two points in $V$ satisfying that $\alpha\cap\alpha'$ is either empty or a point in the endpoints for any pair of Jordan curves $\alpha$ and $\alpha'$.  Any embedding of the topological realization of $G$ into $\mathbb{S}^2$ partitions the 2-sphere into simply connected regions of $\mathbb{S}^2\setminus G$ called the {\em faces} $F$ of the embedding.

Let us define the {\em $d$-antipodal} function

$$\begin{array}{llcc}
\alpha_d: & \mathbb{S}^d& \rightarrow & \mathbb{S}^d\\ 
& x & \mapsto & -x
\end{array}$$

Notice that $\alpha_d$ is an homeomorphism of $\mathbb{S}^d$ into itself without fixed points. We say that $Y\subseteq\mathbb{S}^d$ is {\em $d$-antipodally symmetric} if $\alpha_d(Y)=Y$.
\smallskip

A self-dual map $G$ is called {\em antipodally self-dual} if the dual map $G^*$ is antipodally embedded in $\stw$ with respect to $G$, that is, $\alpha_2(G)=G^*$. We say that $G$ is {\em antipodally symmetric} map if it admits an embedding in $\stw$ such that $\alpha_2(G)=G$.
\smallskip


Let $G$ be a antipodally symmetric map. If $v\in V(G)$ then its {\em antipodal} vertex is given by $\alpha_2(v)=-v$. We call them {\em antipodal pair} of vertices. 

\begin{remark}\label{rem:2-aut} If $G$ is antipodally symmetric map then its number of faces must be even. Moreover, the function $\alpha_2$ naturally matches the pairs of {\em antipodal faces}, say $f$ and $\alpha_2(f)$ (we may refer $\alpha_2(f)$ as the {\em $f$-antipodal} face of $f$). The latter naturally induces a permutation of the faces that turns out to be an automorphism of $G^*$ (that is, $\alpha_2\in Aut(G)$, and thus $G^*$ is also antipodally symmetric). 
\end{remark}

\begin{proposition} \cite[Proposition 1]{MRAR2} Let $G$ be an antipodally symmetric map where its faces are 2-colored properly (that is, two faces sharing an edges have different colors). Then, if one pair of antipodal faces have the same (resp. different color) then all pairs of antipodal faces have the same (resp. different color). 
\end{proposition}

A {\em bicolored map} is a map  $G=(V,E,F)$ together with a coloring $C_X:X\rightarrow \{black, white\}$ where $X$ is either $V(G), E(G)$ or $F(G)$. A {\em signed map} is map $G=(V,E,F)$ together with a signature $S_Y:Y\rightarrow \{+,-\}$ where $Y$ is either $V(G), E(G)$ or $F(G)$. 
\smallskip

Throughout the paper, we will consider bicolored signed maps $(G,C_X,S_Y)$, that is, maps together with both a {\em vertex-, edge-} or {\em face-coloring} $C_X$ and a {\em vertex-, edge-} or {\em face-signature} $S_Y$.
\smallskip

Let $(G,C_F,S_V)$ be a colored-face vertex-signed map. We say that an automorphism $\sigma(G)\in Aut(G)$ is \textit{color-preserving} (resp. \textit{color-reversing}) if each pair of faces $f$ and $\sigma(f)$ have the same (resp. different) color. Similarly, $\sigma$ is said to be \textit{sign-preserving} (resp. \textit{sign-reversing}) if each pair of vertices $v$ and $\sigma(v)$ have the same (resp. different) sign.

\begin{remark}\label{re;anti} In the case when $(G,C_F,S_V)$ is an antipodally symmetric map, the automorphism $\alpha_2$ can be either color-preserving or color-reversing, see Figure \ref{fig14}.
\end{remark}

\begin{figure}[H]
\centering
\includegraphics[width=.6\linewidth]{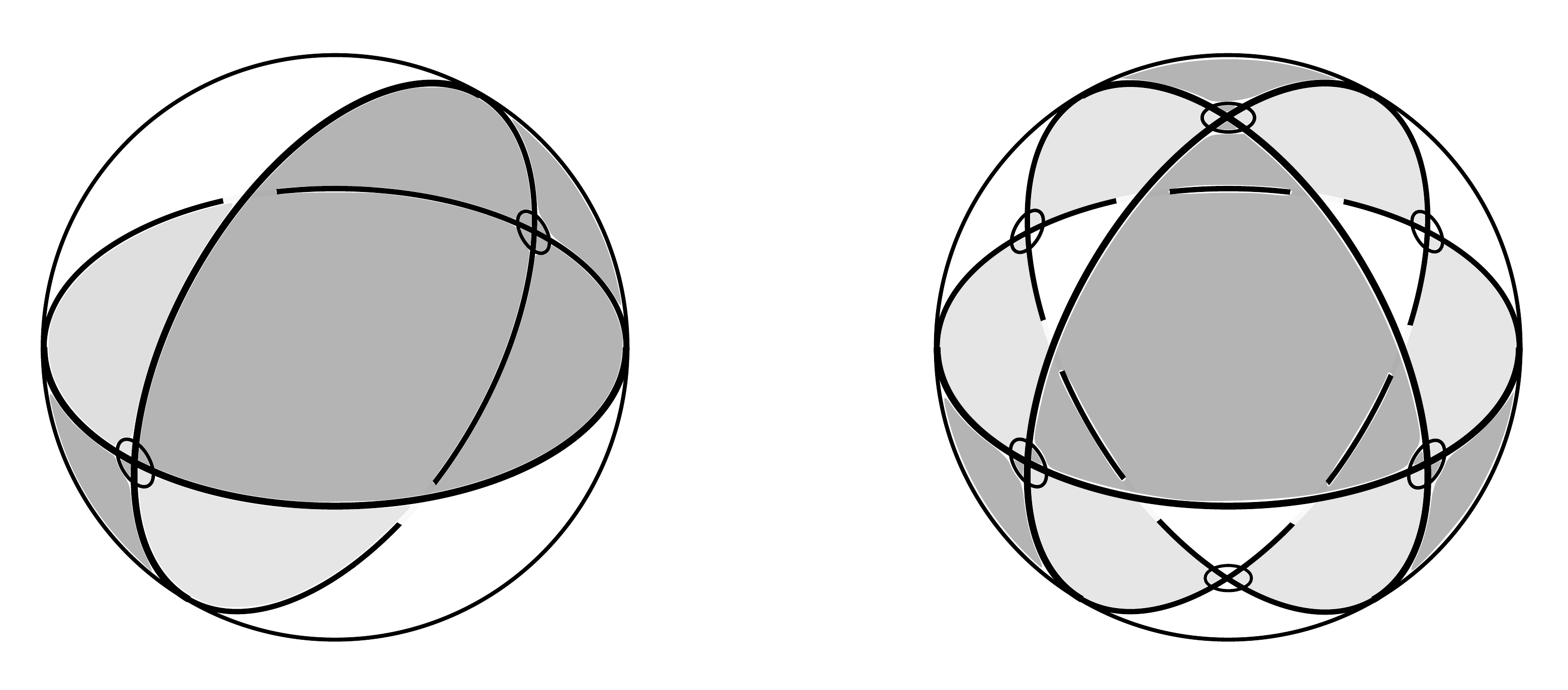}
\caption{(Left) A $2$-antipodally symmetric map where the antipodal mapping is color-preserving. (Right) A $2$-antipodally symmetric map where the antipodal mapping is color-reversing.}\label{fig14}
\end{figure}

\subsection{Special embedding construction}
Recall that the {\em medial graph} of $H$, denoted by $med(H)$, is the graph obtained by placing one vertex on each edge of $H$ and joining two vertices if the corresponding edges are consecutive on a face of $H$. We notice that $med(H)$ is 4-regular since each edge is shared by exactly two faces.
\smallskip

We have that $(med(G), C_F, S_V)$ determines (in a canonical way) the link diagram $D(G,S_E)$. We shall construct a specific embedding of $D(G,S_E)$ in $\rth$, denoted by $\embed(G,S_E)$ by modifying (locally) the diagram around each crossing as follows. Take a small sphere $\mathbb{S}_1$ around each crossing (say, with center the crossing itself). and move (locally) the piece of arc of the diagram passing over (resp. passing under) around $\mathbb{S}_1$ outside (resp. inside) of $\mathbb{S}^2$ according with the {\em crossing sphere rules}, see Figure \ref{crossing-spheres}. 

 \begin{figure}[H]
    \centering
    \includegraphics[width=1\textwidth]{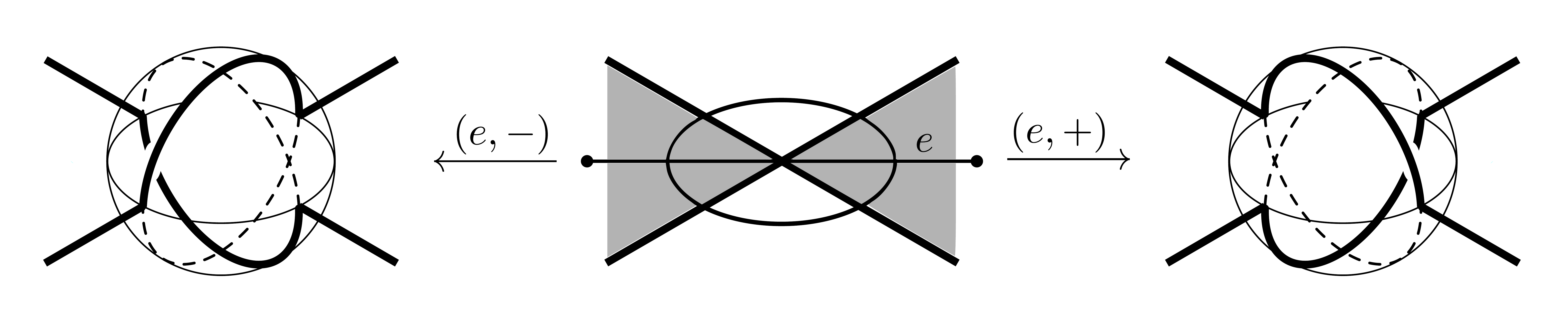}
    \caption{The crossing spheres rules.}
    \label{crossing-spheres}
\end{figure} 

The rest of the diagram $D(G,S_E)$ remains the same in $\mathbb{S}^2$; see Figure \ref{fig16}.

\begin{figure}[H]
\centering
\includegraphics[width=.6\linewidth]{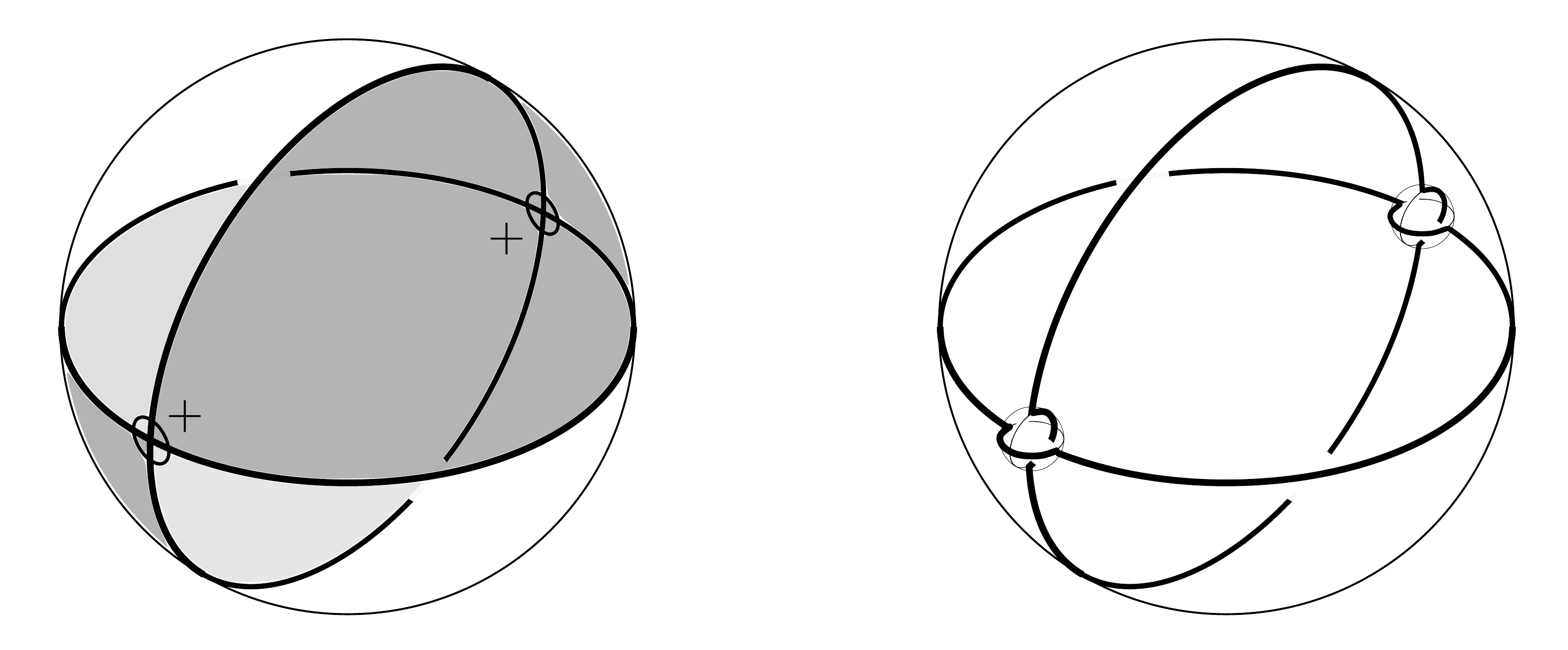}
\caption{(Left) An antipodally symmetric map of the graph $H$ on 2 vertices and four parallel edges. This graph corresponds to a shadow of the Hopf link, see Figure \ref{fig4}. In this case, $\alpha_2$ is color-preserving and sign-preserving. (Right) $\embed(2_1)$.}
\label{fig16}
\end{figure}


\section{A novel approach}\label{sec;newapproach} 

We notice that since $\mathbb{R}\mathbb{P}^3$ is given by the quotient $\mathbb{S}^3/(x=\alpha_3(x))$ then we have that a 3-antipodally symmetric link $L$ is a projective link. 

Let us define the {\em inversion} function

$$\begin{array}{lllc}
i: & \widehat{\mathbb{R}}^3& \longrightarrow & \widehat{\mathbb{R}}^3\\ 
& x & \mapsto & \frac{x}{\|x\|^2}\\
& 0 & \mapsto & \infty\\
& \infty & \mapsto & 0
\end{array}$$

Notice that 

$$\|i(x)\|\left\{\begin{array}{ll}
>1 & \text{ if } \|x\|<1,\\ 
=1 & \text{ if } \|x\|=1,\\ 
<1 & \text{ if } \|x\|>1.\\ 
\end{array}\right.$$

It turns out that a projective link can be interpreted as a link embedded in $\mathbb{R}^3$ invariant under a negative inversion. We say that a link $L$ is {\em anti-inversely symmetric} if it admits an embedding in $\rth$ such that $-i(L)=L$.

\begin{lemma}\label{theo:inversion} Up to isotopy, the set of projective links in $\mathbb{R}\mathbb{P}^3$ is in bijection to the set of anti-inversely symmetric links in $\mathbb{R}^3$.
\end{lemma}

\begin{proof} Recall that $\mathbb{S}^3$ can be  thought as the {\em 1-point compactification} of $\mathbb{R}^3$, that is, we take $\mathbb{R}^3$ and an additional point denoted by $\infty$. By using the {\em stereographic projection}, 
$$\phi: \mathbb{S}^3\setminus (1,0,0,0) \rightarrow \mathbb{R}^3$$ 
it can be showed that $\widehat{\mathbb{R}}^3=\mathbb{R}^3\cup\{\infty\}$ is equivalent to $\mathbb{S}^3$ where the North pole of $\sth$ is mapped to {\em infinity}.  

We claim that the following schema holds

$${\large \xymatrix{\mathbb{S}^3  \ar[r]^{\alpha_3} & \mathbb{S}^3 \ar[d]^\phi \\
\widehat{\mathbb{R}}^3 \ar[u]^{\phi^{-1}} \ar[r]_{-i} & \widehat{\mathbb{R}}^3 }}$$

Indeed, it is known that the stereographic projection from  $\mathbb{S}^3$ to the {\em equatorial plane} (the plane containing the equator of  $\mathbb{S}^3$ in $\mathbb{R}^4$) is given by 
$$\begin{array}{llll}
\phi: & \mathbb{S}^3\setminus (1,0,0,0) & \longrightarrow & \mathbb{R}^3\\ 
& (x,y,z,w) & \mapsto & (\frac{x}{1-w},\frac{y}{1-w},\frac{z}{1-w})\\
\end{array}$$


Let $(x,y,z,w)\in\sth$. Since $x^2+y^2+z^2+w^2=1$ then 

$$\begin{array}{ll}
 -i(\phi(x,y,z,w))& =-i((\frac{x}{1-w},\frac{y}{1-w},\frac{z}{1-w})) \\
 &=-\frac{1-w}{x^2+y^2+z^2}(x,y,z)\\
&=-\frac{1-w}{1-w^2}(x,y,z)\\
&=-\frac{1-w}{(1-w)(1+w)}(x,y,z)\\
&=-\frac{1}{(1+w)}(x,y,z)\\
&=\frac{1}{(1+w)}(-x,-y,-z)\\
&=\phi(-x,-y,-z,-w).
\end{array}$$

We thus have that $$-i\circ \phi(x)=\phi(-x)=\phi\circ\alpha_3(x),$$ and the above diagram follows.

Therefore, if $L$ is a projective link then it verifies that $\alpha_3(L)=L$ and, by the above, we obtain $$-i\circ \phi(L)=\phi\circ\alpha_3(L)=\phi(L).$$
\end{proof}

We may now present our combinatorial characterization.

\begin{theorem} \label{thm:antipodalS3} 
		 A link $L$ in $\widehat{\mathbb R^3}$ is anti-inversely symmetric if and only if
		 there is an edge-signed map $(G,S_E)$ in $\mathbb S^2$ satisfying the following conditions:
		 
		 \begin{enumerate}
		 	\item $L$ is isotopically equivalent to $L(G,S_E)$.
		 	\item The medial of $G$ is antipodally symmetric in $\mathbb S^2$.
		 	\item The face-colored and vertex-signed map $(med(G),C_F,S_V)$ induced by $(G,S_E)$ satisfies that the antipodal function $\alpha_2$ is either color and sign-preserving, or color and sign-reversing.
		 \end{enumerate}
		 
	\end{theorem}

\begin{proof}  ({\em Sufficiency}) Let $(med(G),C_F,S_V)$ be an antipodal symmetric medial map realized by $\alpha_2$. We consider the embedding $\embed(G,S_E)$. It can be checked that if $\alpha_2$ is either color-, sign-preserving or color-, sign-reversing then the piece of arc of the diagram passing over (resp. passing under) around vertex $v$ correspond to the piece of arc of the diagram passing under (resp. passing over) around the antipodal  vertex $\alpha_2(v)$, see Figure \ref{fig:2antisym}

\begin{figure}[H]
\centering
\includegraphics[width=0.8\linewidth]{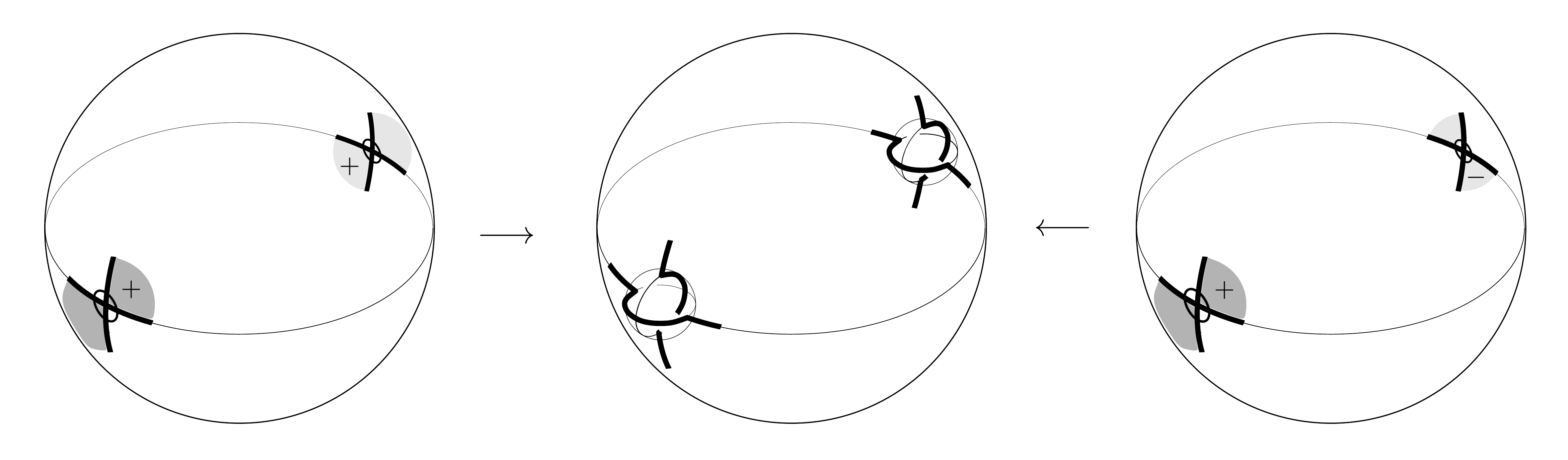}
\caption{(Left) Antipodal pair of vertices of $med(G)$ in case (a). (Right)  Antipodal pair of vertices of $med(G)$ in case (b) (Center) The local modifications around the antipodal pair of vertices following the color-sphere rules.}
\label{fig:2antisym}
\end{figure}

 We thus have that  $\embed(G,S_E)$ is anti-inversely symmetric. 
\smallskip

 ({\em Necessity}) Suppose that $L$ is anti-inversely symmetric. Hence, $L$ admits an embedding, say  $\hat L$ in $\widehat{\mathbb R^3}$ with $-i(\hat L)=\hat L$. We claim that $\hat L$ can be thought of as an special embedding  $\embed(G,S_E)$ for some map $G$. For this, we first  take the radial projection $p(\hat L)$ from $\hat L$ to $\stw$, that is, if we let $r(x)$ be the ray emitting from the origin passing through $x$ then

$$\begin{array}{lllc}
p: & \hat  L& \longrightarrow & \stw\\ 
& x & \mapsto & r(x)\cap\stw
\end{array}$$

Since $L$ is anti-inversely symmetric then $p(\hat  L)$ is clearly antipodally symmetric in $\stw$ (realized by $\alpha_2$). We suppose that $p(\hat  L)$ avoids cusp and tangency point (this can be obtained by making some suitable local modifications to $p(\hat  L)$ done in a symmetric fashion in order to keep the symmetric antipodality of $p(\hat  L)$). 
\smallskip

Let $y$ be a point in $p(\hat L)$, suppose that
$$y=p(x_1)=\cdots =p(x_k)$$ with $||x_1||<\cdots <||x_k||$, that is, $x_i$ is nearer to the origin than $x_{i+1}$ for each $1\le i < k$. 

We say that $y$ is {\em simple intersection} if $k=2$ and {\em multiple intersection} if $k\ge 3$.
 
We notice that $$\alpha_2(y)=p(-i(x_1))=\cdots =p(-i(x_k))$$
with $||-i(x_1)||>\cdots >||-i(x_k)||$, that is, $-i(x_i)$ is further from the origin than $-i(x_{i+1})$ for each $1\le i < k$. 
\smallskip

If $k=2$ we may write $y_{\{x_1,x_2\}}$ to insist that $y$ arises from the projection of the pieces of $\hat L$ containing $x_1$ and $x_2$.  
\smallskip

If $k=3$, that is, $y$ is a multiple (triplet) intersection, we modify $p(\hat L)$ around $y$ by moving (slightly) the piece of $p(\hat L)$ containing $p(x_{3})$ to avoid the intersection $y_{\{x_1,x_2\}}$ (as well as any other multiple or simple intersection). We do the same symmetrically for $\alpha_2(y)$. We notice that this create new simple intersections, say $y_{\{x_1,x_3\}}$ and $y_{\{x_2,x_3\}}$ (and symmetrically $\alpha_2(y_{\{x_1,x_3\}})=y_{\{-i(x_1),-i(x_3)\}}$ and $\alpha_2(y_{\{x_2,x_3\}})=y_{\{-i(x_2),-i(x_3)\}})$).  These modifications can be seen as moving (slightly) $\hat L$ in order to obtain a new representation (isotopic to $\hat L$) such that its radial projection agrees with the modifications realized to $p(\hat L)$. We carry on this procedure if $k\ge 4$ and for all multiple intersections.  We clearly end up with an antipodally symmetric projection without multiple points (only with simple ones). This projection can be thus thought of as a 4-regular antipodally symmetric map (realized by $\alpha_2$) which, in turn, can be regarded as a medial map $med(G)$ for some map $G$. 
\smallskip

Let $y_{\{x_\ell,x_{\ell'}\}}$ be a simple intersection with $\ell<\ell'$. In such a case, we have the information that the piece of $p(\hat L)$ containing $x_{\ell'}$ pass over the piece containing $x_\ell$. Notice that for $\alpha_2(y)$ the piece of $p(\hat L)$ containing $-i(x_{\ell})$ pass over the piece containing $-i(x_{\ell'})$ and thus keeping the negative invertibility. 
\smallskip
  
Now, if we color the faces of $med(G)$ and sign its vertices (simple intersections) according with the latter information and respecting the crossing sphere rules then we have that the induced antipodally symmetric color-face vertex sign $(med(G),C_F,S_V)$ induce a link isotopic to $L$. Therefore, by construction, $\embed(G,S_E)$ is a anti-inversely symmetric link. Moreover, the only way that the over/under crossing for each antipodal pair of vertices verify the negative invertibility is when $\alpha_2$ is either color-,sign-preserving or color-,sign-reversing.
\end{proof}
In Appendix \ref{appendix}, we present a table containing the incident graphs, symmetric cycles, Tait and medial graphs of the first 14 nontrivial projective links.
\smallskip

Figure \ref{fig23} illustrates a map $(med(G),C_F,S_V)$ (inducing a link $L$) and the embedding $\embed(L)$ such that $-i(L)=L$. 

\begin{figure}[!htp]
\centering
\includegraphics[width=0.9\linewidth]{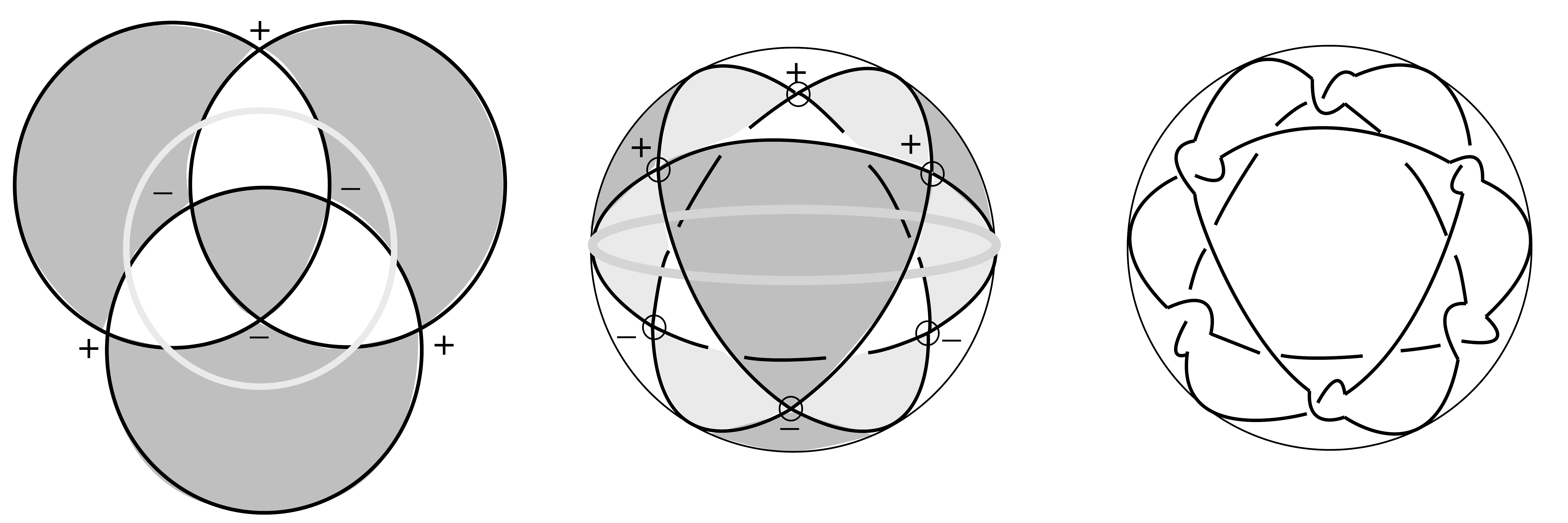}
\caption{ (Left) Stereographic projection in the equatorial plane (the equatorial is represented by a gray circle)  the corresponding diagram in $\mathbb{R}\mathbb{P}^2$ is given by the circle and the part of the projection inside of it.
(Center) antipodally symmetric face-colored and vertex-signed $med(G)$ (Right) anti-inversely symmetric embedding $\embed(med(G),C_F,S_V)$.}
\label{fig23}
\end{figure}

\subsection{Diagram in $\mathbb{R}\mathbb{P}^2$}\label{subsec:diagram} In a similar fashion, as explained in Subsection \ref{subsect:pro}, an anti-inversely symmetric link $L$ can also be regarded as a diagram in $\mathbb{R}\mathbb{P}^2$. Indeed, it suffice to take the stereographic projection of such embedding to the equatorial plane, Figure \ref{fig23} (Left) and to restrict ourselves to the projection of the lower hemisphere of  $\mathbb{S}^2$ (inside the circle) and merely identify antipodal points on the bounding equator, obtaining the desired diagram in $\mathbb{R}\mathbb{P}^2$. 
\smallskip

As remarked above, it may happens that two different diagrams in $\mathbb{R}\mathbb{P}^2$ may lead
to the same projective link. We can see this, for instance, by choosing a different plane for the stereographic projection. For example, in Figure \ref{fig23}, if we project to the plane containing the sheet, the obtained diagram in $\mathbb{R}\mathbb{P}^2$ have different vertex signs, however they both induce the same projective link, see Figure \ref{fig23aa}.

\begin{figure}[!htp]
\centering
\includegraphics[width=0.7\linewidth]{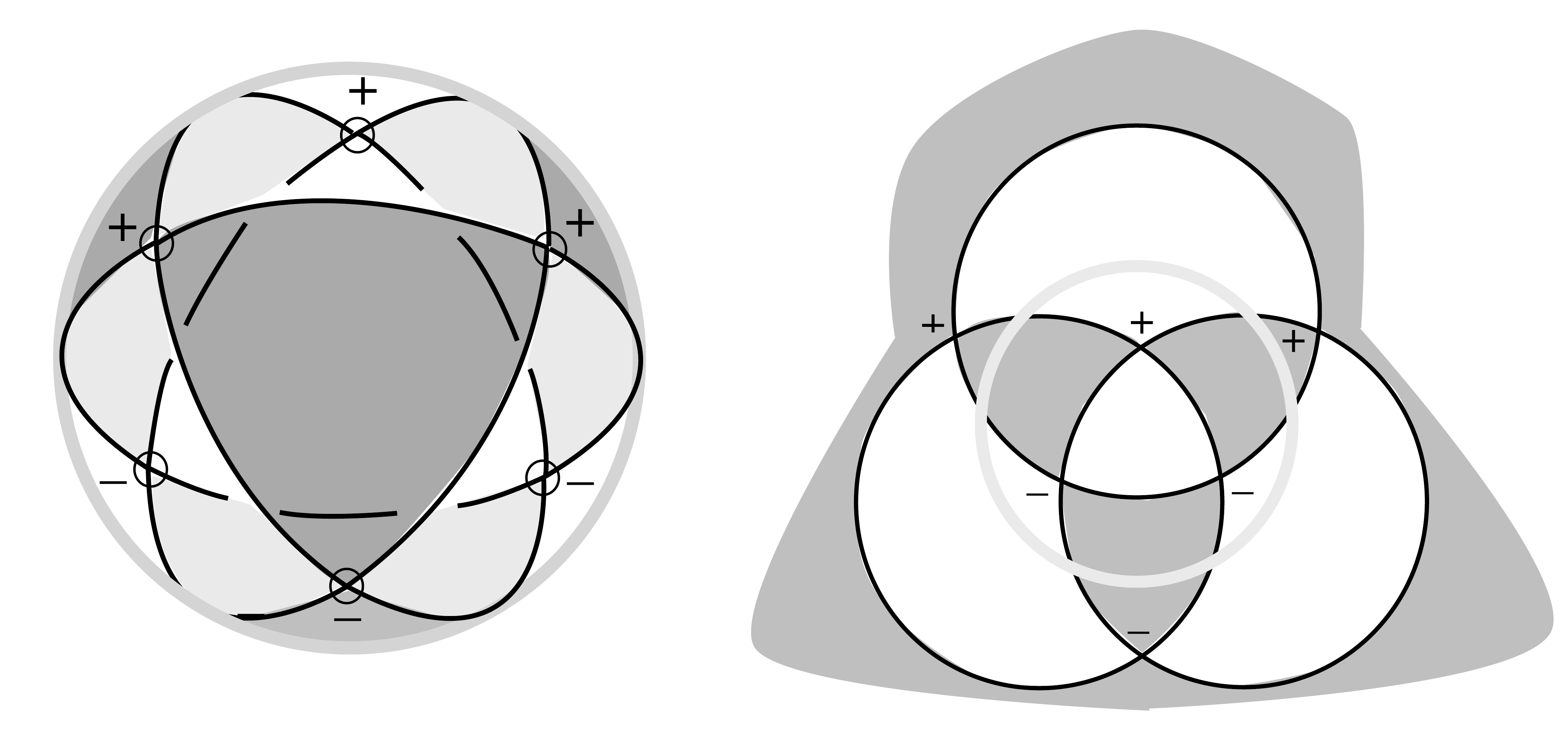}
\caption{Diagram in $\mathbb{R}\mathbb{P}^2$ obtained by taking the stereographic projection on the plane containing the sheet (gray circle).}
\label{fig23aa}
\end{figure}

\subsection{Antopidally symmetric}\label{subsec:antisym}  In \cite[Lemma 1]{MRAR1} was proved that if $G$ is an antipodally self-dual map then $med(G)$ is antipodally symmetric. Therefore, if $G$ is an antipodally self-dual then $med(G)$ is a good candidate to construct a projective link. However, not all projective links arise from antipodally self-dual maps, see Table \ref{tab:proj1}. Indeed, even if $med(G)$ is antipodally symmetric it also needed proper face-coloring and vertex-signature (as in Theorem \ref{thm:antipodalS3}). If $G$ is antipodally self-dual then this implies that $\alpha_2$ is always color-reversing (see Corollary below) and therefore we also need that $\alpha_2$ to be sign-reversing in order to induce a projective link (which is not always the case).

\begin{corollary}\label{rem;color} If $G$ is an antipodally self-dual map (realized say by $\alpha_2$) then $\alpha_2$ is color-reversing. 
\end{corollary}

\begin{proof} Since $G$ is antipodally self-dual then antipodal faces of $med(G)$ correspond to a pair of antipodal vertices, say $v\in V(G)$ and $\alpha_2(v)\in V(G^*)$. Therefore, if we color all faces corresponding to vertices in $G$ (resp. in $G^*$) in black (resp. in white) we obtain that  $\alpha_2$ is color-reversing. 
\end{proof}

To determine if a given link is projective seems a difficult task. On this direction, the following problem (interest in its own) could be a first natural step.  

\begin{problem} Find necessary and sufficient conditions for a map to be antipodally symmetric.
\end{problem}

\section{Symmetric cycles and alternating links}\label{sec:conclud}

\subsection{Symmetric cycle}

A nice feature of our approach is that the stereographic projection of an anti-inversely symmetric link gives a natural symmetry with respect to the projection of the projected equatorial. The latter can be nicely interpreted in terms of the {\em incident} graph.
\smallskip

Let $G$ be a plane graph and $G^*$ its geometric dual. We recall that  the {\em (vertex-face) incidence} graph $I(G)=(V^{I},E^{I})$  has as vertices $V^{I}=V(G)\cup V(G^*)$ and two vertices $v\in V(G)$ and $v^*\in V(G^*)$ are adjacent if $v$ is a vertex of the face corresponding to $v^*$,  see Figure \ref{fig255}

\begin{figure}[H]
   \centering
   \includegraphics[width=0.3\textwidth]{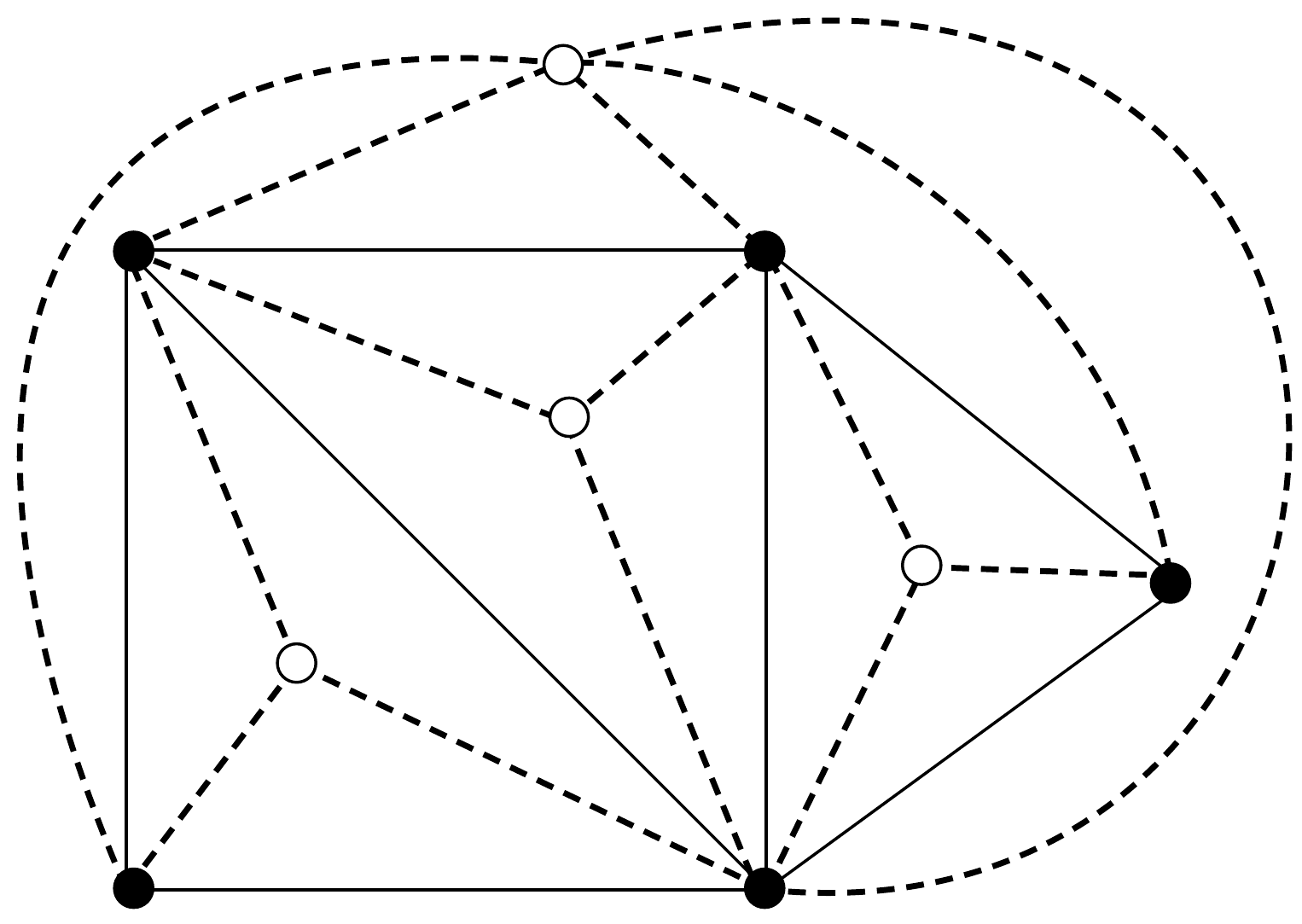}
  \caption{Graph $G$ (straight edges) and its incidence graph $I(G)$ (dashed edges).} \label{fig255}
\end{figure}

A cycle $C$ of a planar graph $G$ is said to be {\em symmetric} if there is an automorphism $\sigma(G)$ such that $\sigma(C)=C$ and $\sigma(int(C))=ext(C)$, that is, the induced graph in the interior of $C$ is isomorphic to the induced graph in the exterior of $C$. 
\smallskip

 It was proved in \cite[Theorem 1]{MRAR1} that if $G$ is an antipodally self-dual map then $I(G)$ always admits at least one symmetric cycle (and all symmetric cycles in $I(G)$ are of length $2n$ with $n\ge 1$ odd).  The proof was based on the fact that $med(G)$ is antipodally symmetric (see  \cite[Lemma 1]{MRAR1}). We may thus apply exactly the same arguments for any antipodally symmetric medial graph $med(G)$ arising from a not necessarily antipodally self-dual map $G$. It turns out that this is the case for anti-inversely symmetric links as explained in the proof of Theorem \ref{thm:antipodalS3}.
 
\begin{proposition}\label{prop:symcyc} Let $L$ be a anti-inversely symmetric link obtained from $med(G)$ for some map $G$. Then, $I(G)$ admits a symmetric cycle. Moreover, all symmetric cycles in $I(G)$ are of length $2n$ with $n\ge 1$.
\end{proposition}

We point out that the eveness of the length of the cycle arise from the antipodal symmetry. However, the oddness of $n$ cannot be assured when $L$ is anti-inversely symmetric.
\smallskip

The necessary condition of Proposition \ref{prop:symcyc} might be useful to detect if a map is a good candidate to induce a projective link.  
However, the existence of a symmetric cycle (even if  $G$ is antipodally self-dual) might not be sufficient to produce a projective link. For example, consider the {\em Borromean rings}, arising from $(K_4,S_E^+)$ where $S_E^+$ means that all the signs are $+$ (this come from the fact that the Borromean rings is an alternating link), see Figure \ref{fig22d}. 

\begin{figure}[H]
\centering
\includegraphics[width=0.6\linewidth]{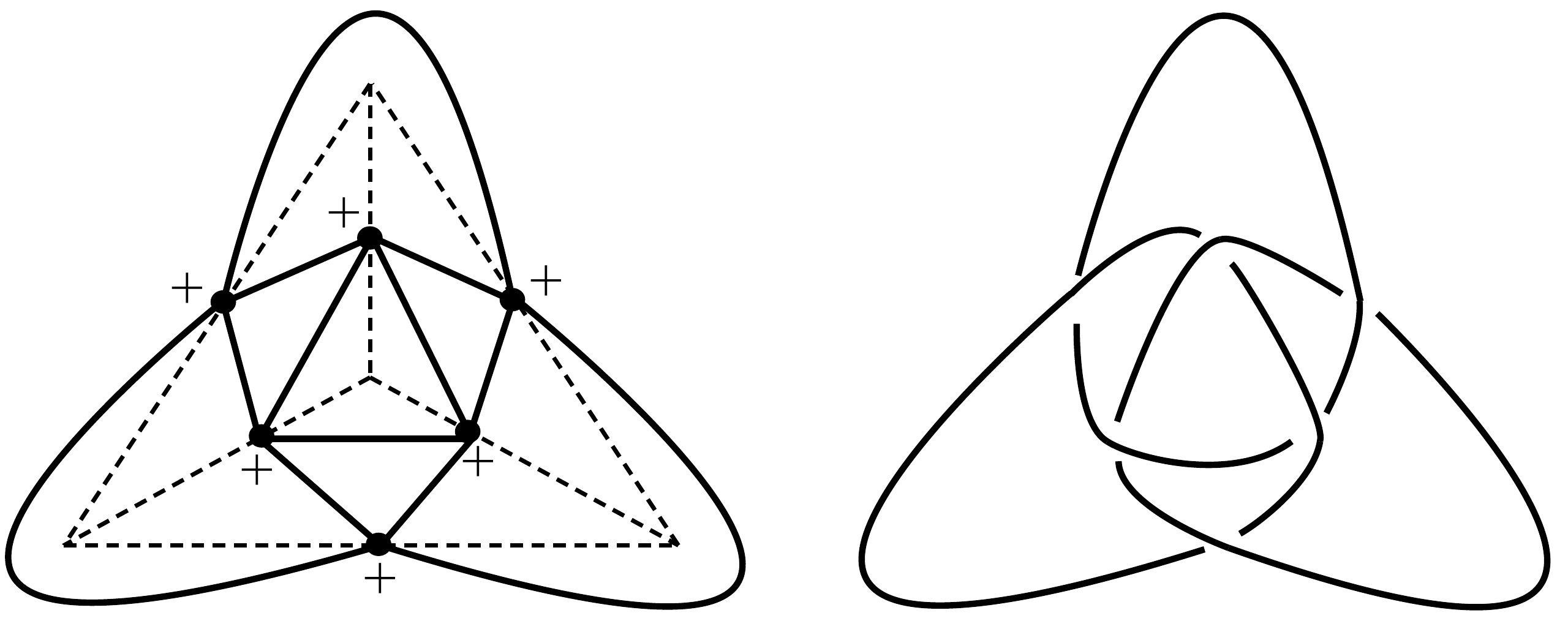}
\caption{(Left) $K_4$ (dashed edges) and $med(K_4)$ (straight edges) (Right) Diagram $(K_4,S_E^+)$.}
\label{fig22d}
\end{figure}

It can be checked that $I(K_4)$ admits a symmetric cycle, see Figure \ref{fig22c}.  

\begin{figure}[H]
\centering
\includegraphics[width=0.25\linewidth]{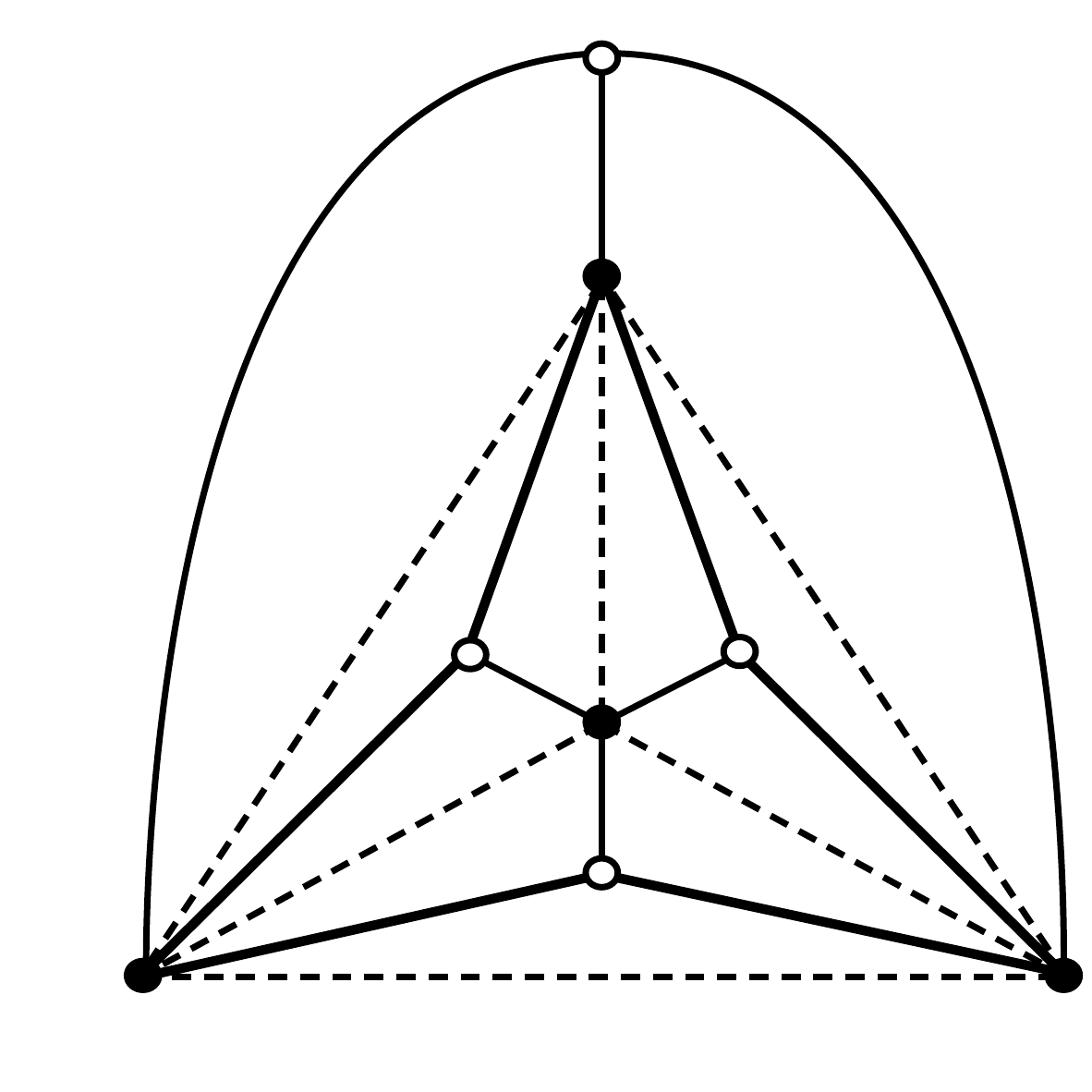}
\caption{$K_4$ (dashed edges), $I(K_4)$ (straight edges) and a symmetric cycle (bold edges).}
\label{fig22c}
\end{figure}

However the Borromean rings is not a projective link. Indeed, by Corollary \ref{rem;color}, $\alpha_2$ is color-reversing and  since all the vertex signs are the same then $\alpha_2$ is
vertex-preserving, see Figure \ref{fig20a}. Therefore, by Theorem \ref{thm:antipodalS3}, the link induced by $(med(K_4), C_F, S_V^+)$ cannot be projective. 

\begin{figure}[H]
\centering
\includegraphics[width=0.3\linewidth]{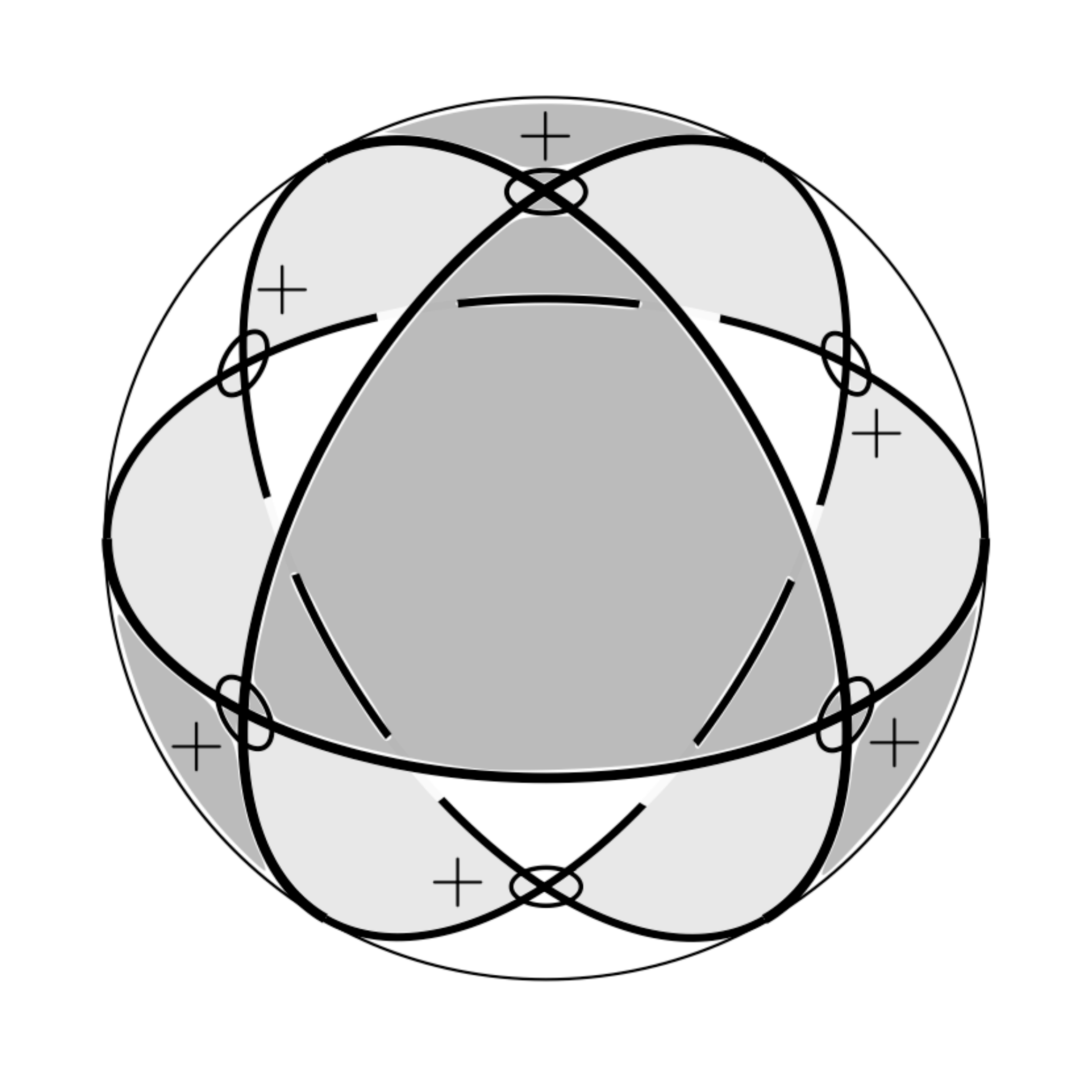}
\caption{An antipodally symmetric, face-colored and vertex-sign map $(med(K_4),C_F,S_V^+)$.}
\label{fig20a}
\end{figure}


\subsection{Alternating}
We say that a projective link is {\em alternating} if it admits an anti-inversely symmetric link having an alternating diagram. 

Symmetric cycles allows us to detect, in some cases, whether a projective link cannot be alternating. By Proposition \ref{prop:symcyc}, if $L$ is anti-inversely symmetric then the symmetric cycle is of length $2n$ with $n\ge 1$ integer. We notice that the parity of $n$ is actually determined by the face-coloring of $med(G)$. It can easily be checked that $n$ is odd (resp. even) if $\alpha_2$ is color-reversing (resp. color-preserving). Therefore, if $L$ is anti-inversely symmetric admitting an antipodally  symmetric cycle $2n$ with $n$-odd (and thus with $\alpha_2$ color-reversing) then $\alpha_2$ must be, by Theorem \ref{thm:antipodalS3}, sign-reversing but this is only possible if $L$ is nonalternating.
\smallskip

We notice that the condition $n$-odd is not sufficient for being nonalternating. For instance, the Hopf link (first projective link in Table \ref{tab:proj1}) admit a symmetric cycle of length 2 (and thus with $n=1$) and it is clearly alternating.
\smallskip

We end this section with the following 

\begin{problem} Characterize alternating projective links.
\end{problem}

\section{Concluding Remarks}

In \cite{Drobo1}, Drobotukhina used the combinatorics of the diagrams in $\mathbb{R}\mathbb{P}^2$ in order to present a classification of projective links with at most 6 crossings. In particular, Drobotukhina gave two nice lemmas \cite[Lemmas 1 and 2]{Drobo1} in which the combinatorics of the induced faces are studied.

\begin{question} Is there a natural translation of these two lemmas into our setting in terms of incidence graphs (and symmetric cycles) ?
\end{question}

In \cite{Drobo}, Drobotukhina introduced a polynomial for oriented  links in $\mathbb{R}\mathbb{P}^3$ generalizing the Jones polynomial (in Kauffman's version) for oriented links in $\mathbb{R}^3$. In  \cite{Drobo1}, it was defined a related Laurent polynomial (on one variable) invariant for nonoriented links in $\mathbb{R}\mathbb{P}^3$. It was remarked that this polynomial coincide with the classical Jones polynomial for a link contained in $\mathbb{R}^3\subset \mathbb{R}\mathbb{P}^3$.

\begin{question} Is there a  connection between the generalized Jones polynomial of a link in $\mathbb{R}\mathbb{P}^3$  and the classical Jones polynomial of the corresponding anti-inversely symmetric representation ?
\end{question}

It is well-known that two nonisotopic links may have the same Jones polynomial; see \cite{EHK} for a nice construction.

\begin{question} Is it possible to have two nonisotopic projective links having the same generalized Jones polynomials ?
\end{question}



\appendix
\section{Projective links diagrams}\label{appendix}

\begin{table}[ht]
\caption{(Left) Projective link $L$ (bold lines), incidence graph (dotted edges) and a symmetric cycle (bold dotted edges). The corresponding projective diagram in $\mathbb{R}\mathbb{P}^2$ is represented by the part of $L$ inside the gray cycle (Right) The corresponding Tait graph $G$ (bold edges) and the 2-colored face $med(G)$.}
\label{tab:proj1}\centering
\begin{tabular}{*{3}{m{0.33\textwidth}}}
\hline
\begin{center}\includegraphics[width=0.8\linewidth]{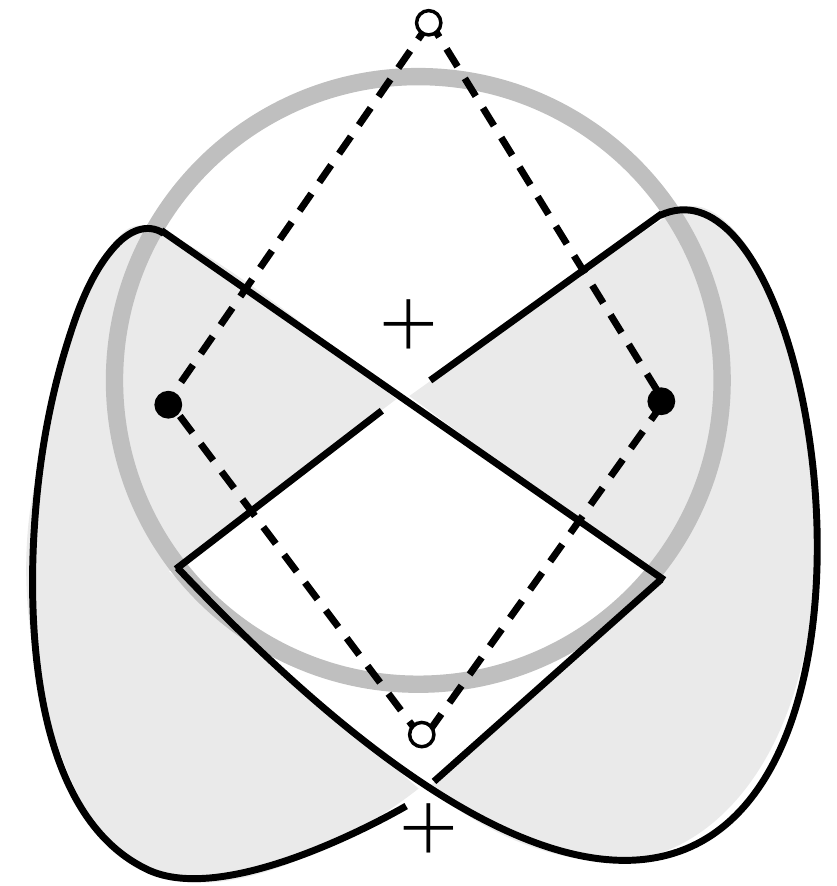}\end{center}&
\begin{center}\includegraphics[width=0.8\linewidth]{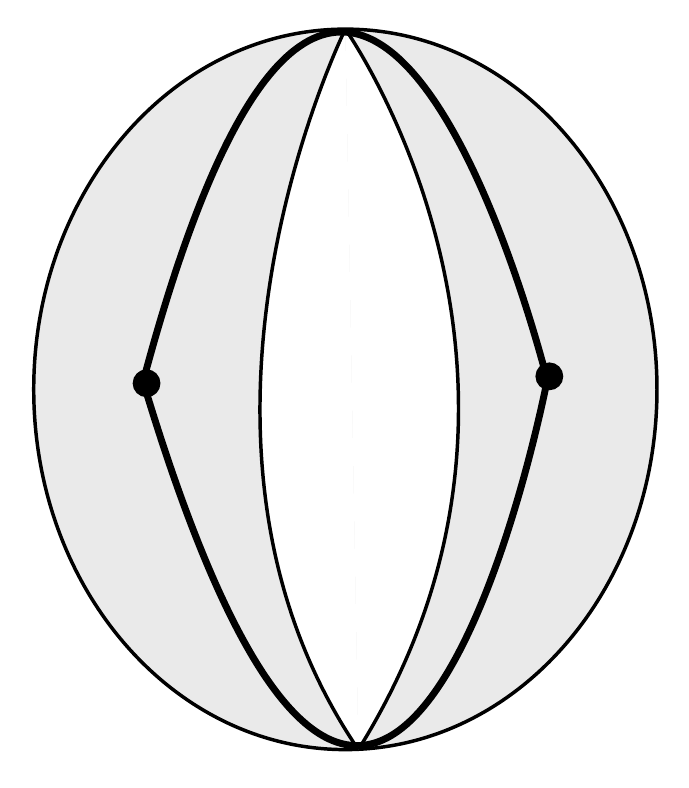}\end{center}&
 \Longstack{$\alpha_2$: color-,sign-preserving \\  $L$: alternating\\  $G$:   self-dual \\  not antipodally self-dual \\ antipodally symmetric 
 }\\
\hline
\begin{center}\includegraphics[width=0.8\linewidth]{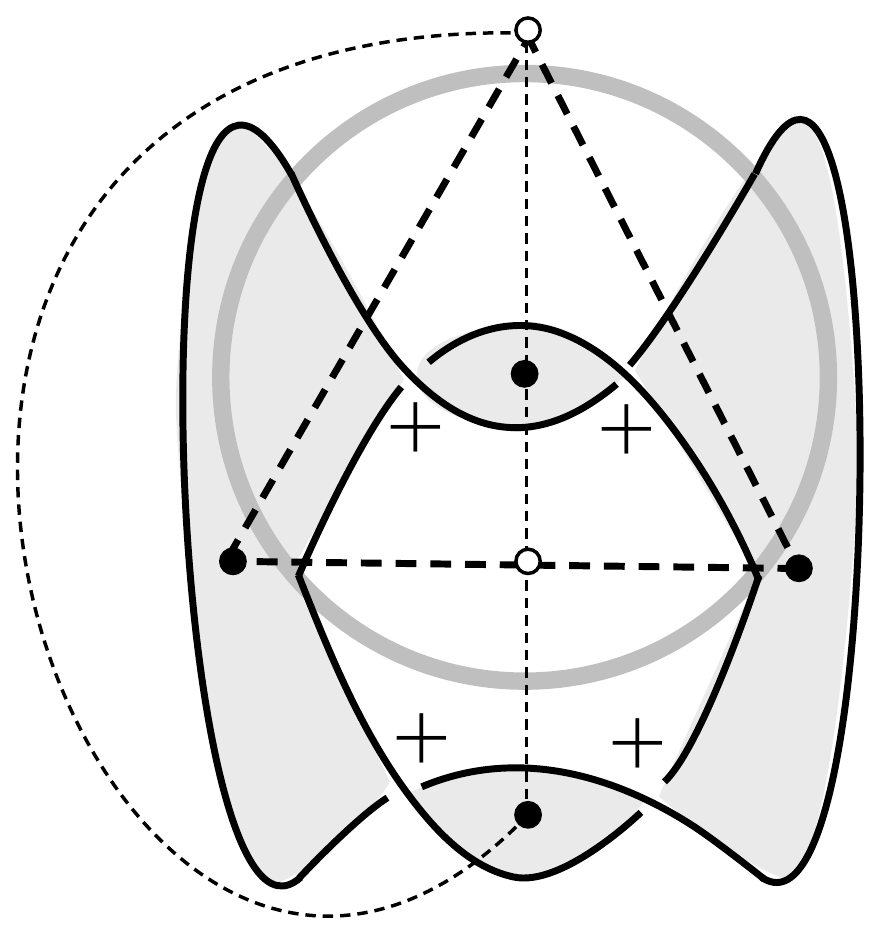}\end{center}&
\begin{center}\includegraphics[width=0.8\linewidth]{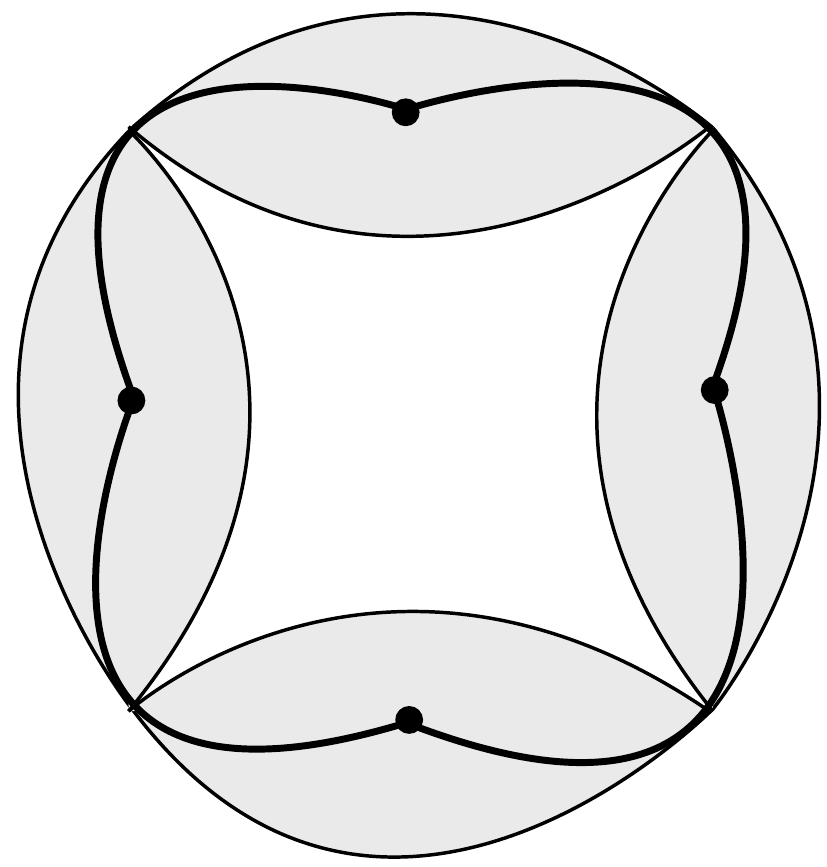}\end{center}&
 \Longstack{ $\alpha_2$: color-,sign-preserving \\ $L$: alternating\\ $G$: not self-dual \\  not antipodally self-dual \\ antipodally symmetric 
 }\\
\hline
\begin{center}\includegraphics[width=0.8\linewidth]{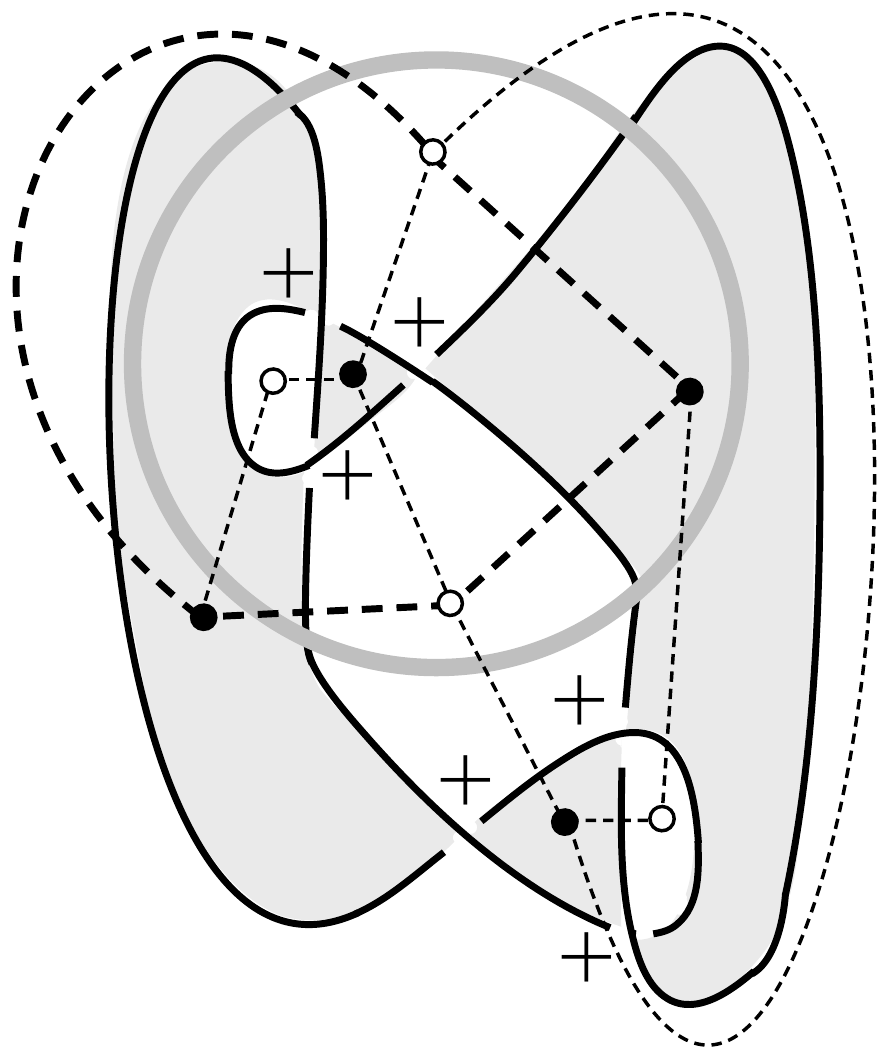}\end{center}&
\begin{center}\includegraphics[width=0.8\linewidth]{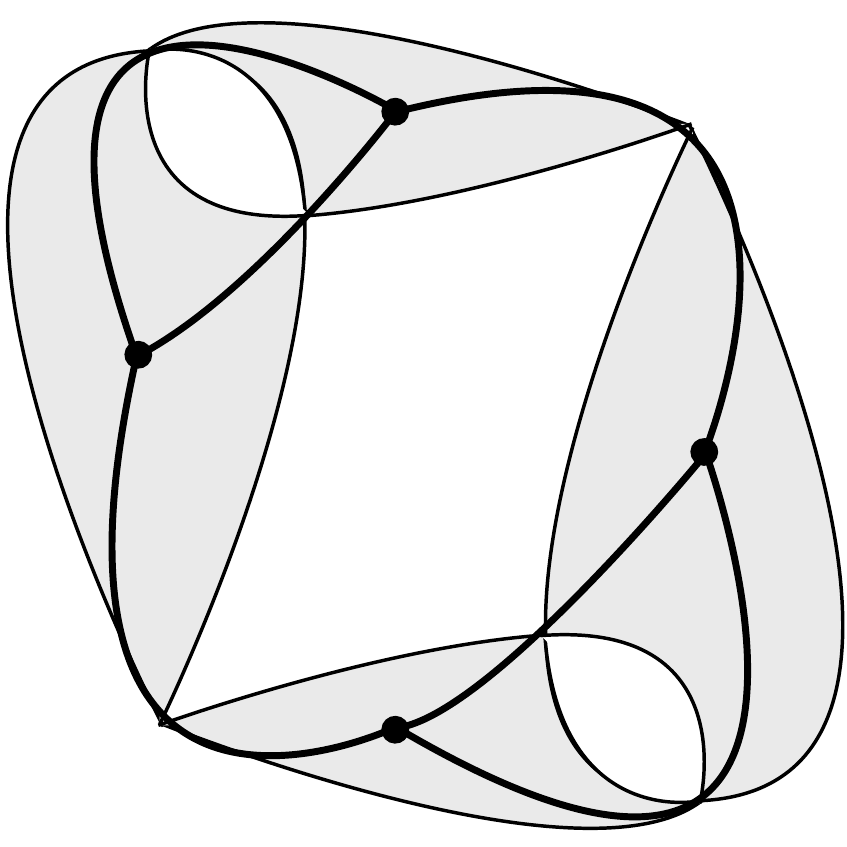}\end{center}&
 \Longstack{$\alpha_2$: color-,sign-preserving \\ $L$: alternating\\ $G$:  not self-dual \\  not antipodally self-dual \\ antipodally symmetric 
 }\\
 \hline
\end{tabular}
\end{table}

\begin{table}[ht]
\centering
\begin{tabular}{*{3}{m{0.33\textwidth}}}
\hline
\begin{center}\includegraphics[width=1.1\linewidth]{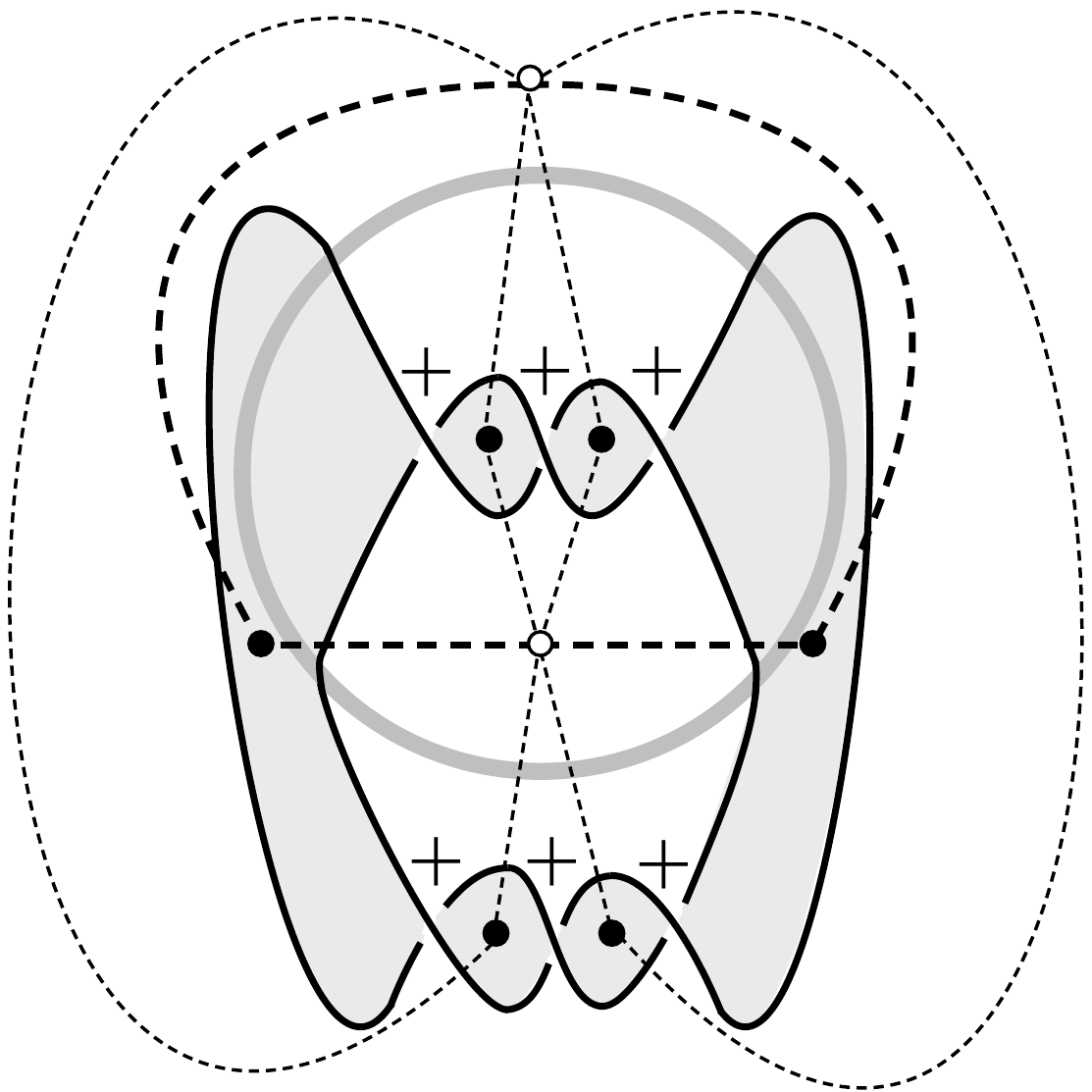}\end{center}&
\begin{center}\includegraphics[width=0.9\linewidth]{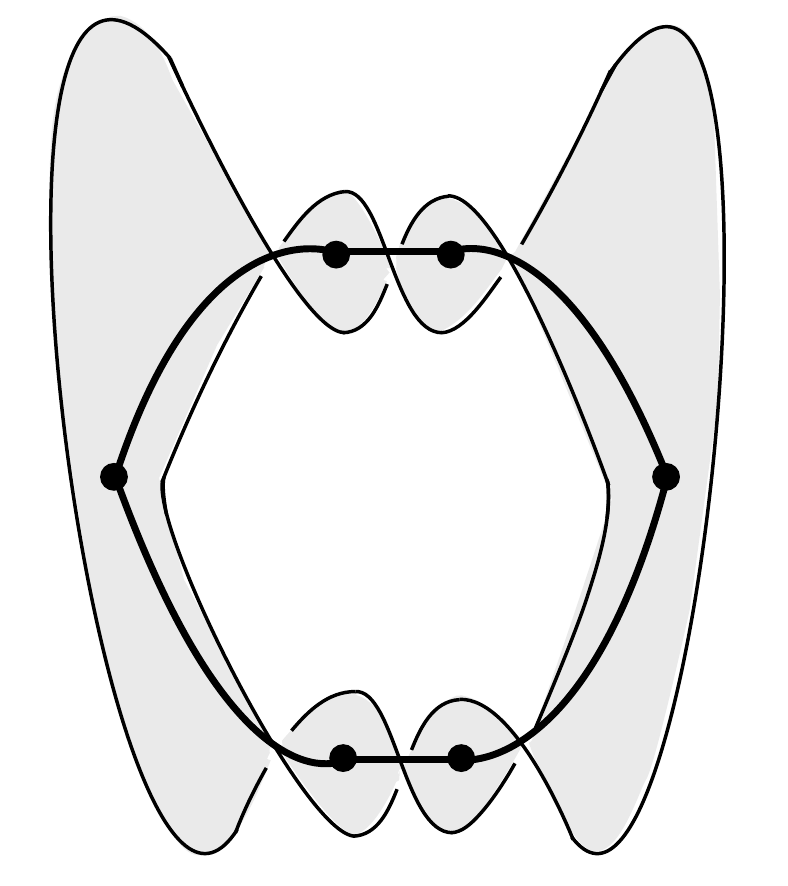}\end{center}&
 \Longstack{$\alpha_2$: color-,sign-preserving \\  $L$: alternating\\  $G$:  not self-dual, \\ not antipodally self-dual, \\ antipodally symmetric 
 }\\
\hline
\begin{center}\includegraphics[width=1.1\linewidth]{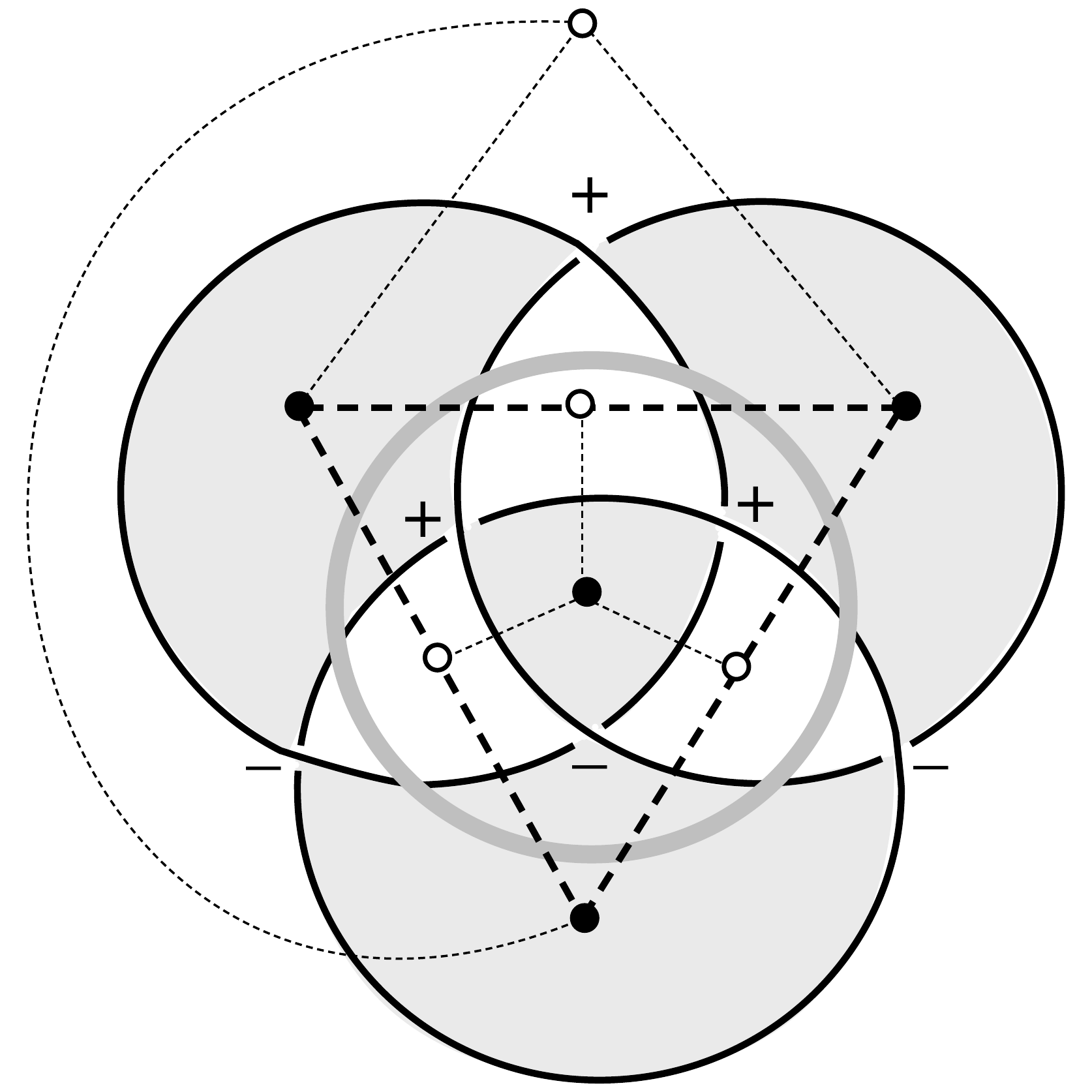}\end{center}&
\begin{center}\includegraphics[width=0.9\linewidth]{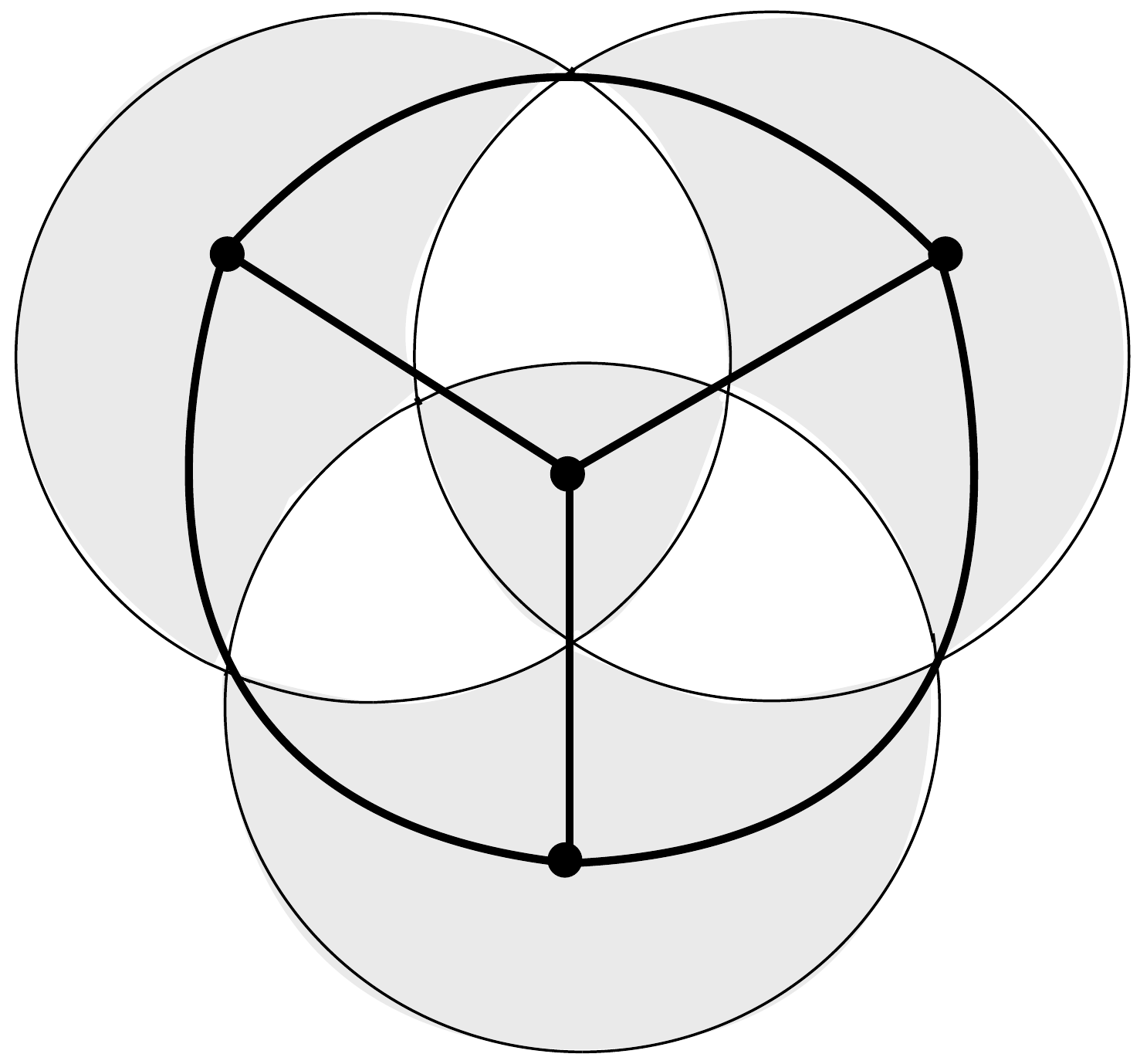}\end{center}&
 \Longstack{$\alpha_2$: color-,sign-reversing \\ $L$: nonalternating\\ $G$:   self-dual \\  antipodally self-dual \\ not antipodally symmetric 
 }\\
\hline
\begin{center}\includegraphics[width=1\linewidth]{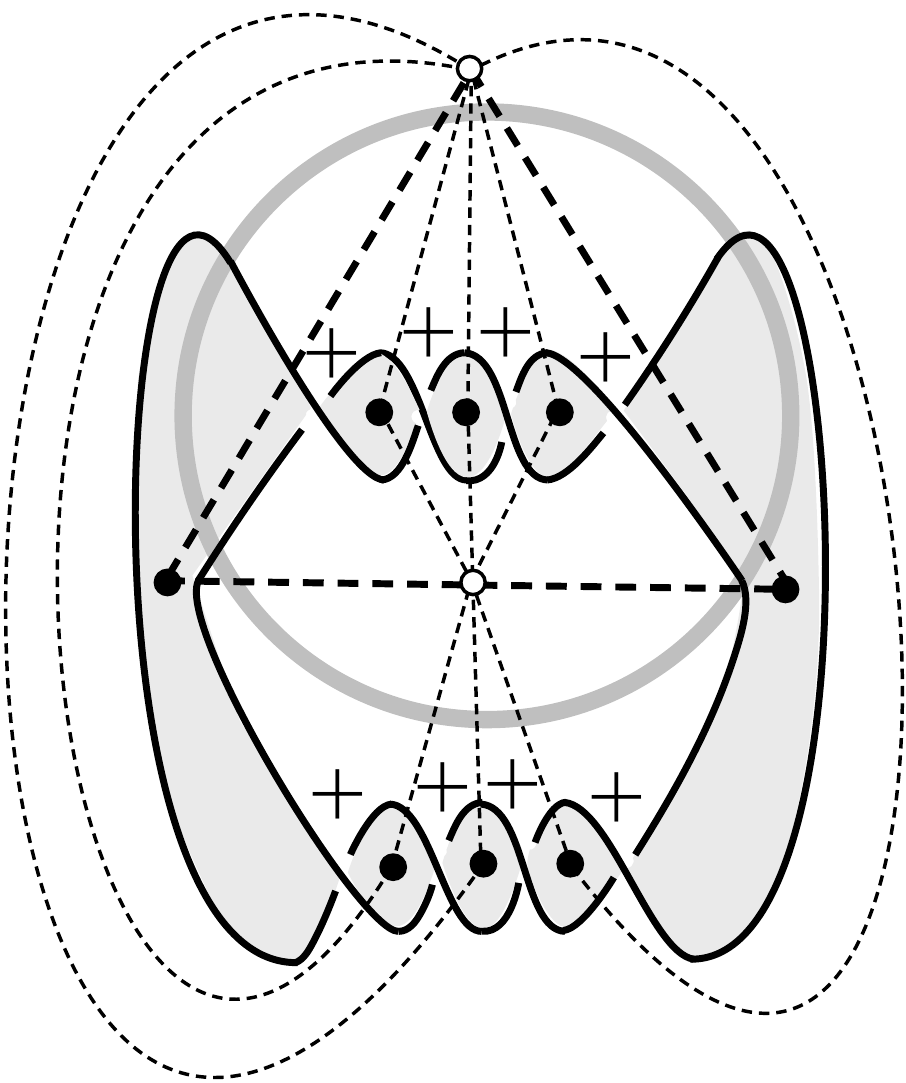}\end{center}&
\begin{center}\includegraphics[width=0.9\linewidth]{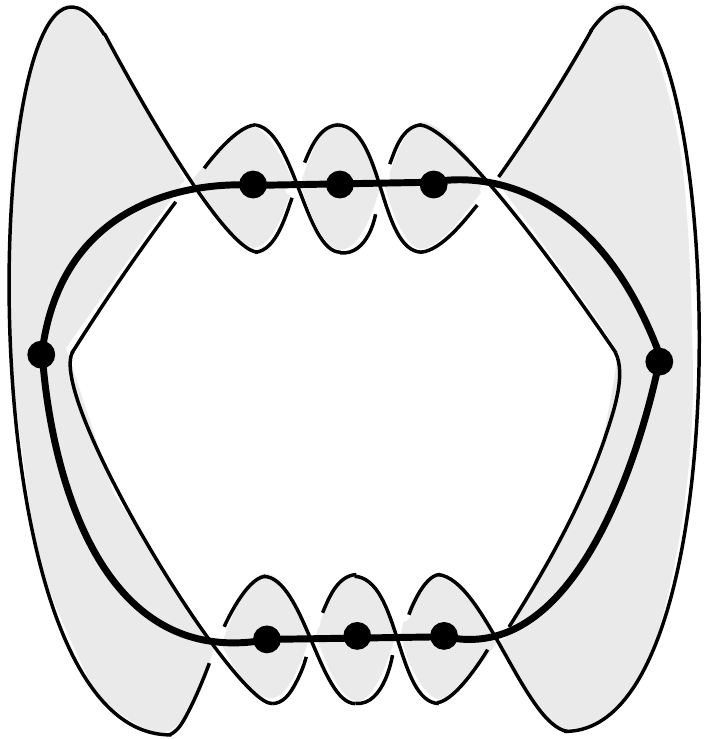}\end{center}&
 \Longstack{ $\alpha_2$: color-,sign-preserving \\ $L$: alternating\\ $G$:  not self-dual \\ not antipodally self-dual \\ antipodally symmetric 
 }\\
 \hline
\end{tabular}
\end{table}

\begin{table}[ht]
\centering
\begin{tabular}{*{3}{m{0.32\textwidth}}}
\hline
\begin{center}\includegraphics[width=1\linewidth]{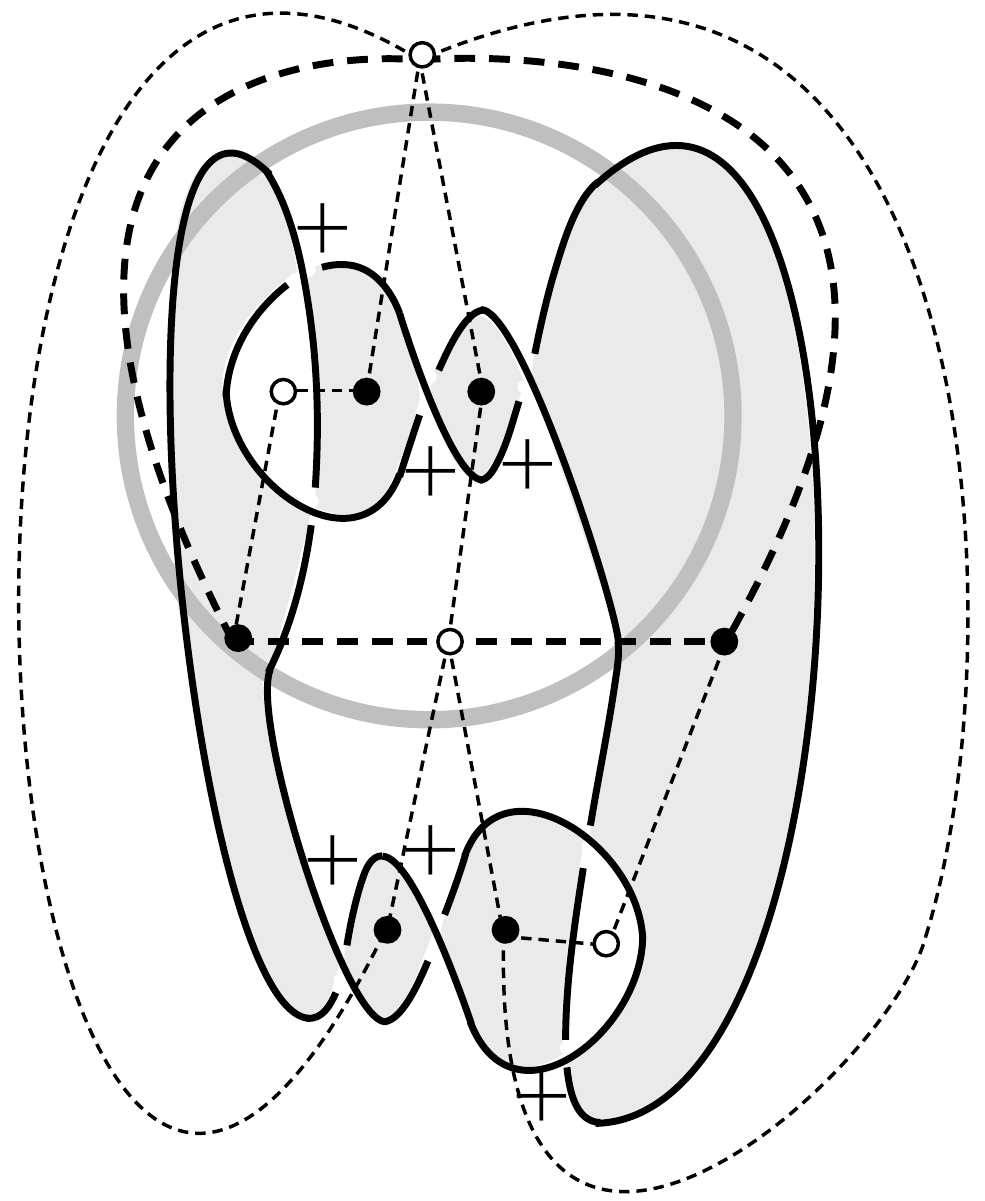}\end{center}&
\begin{center}\includegraphics[width=0.8\linewidth]{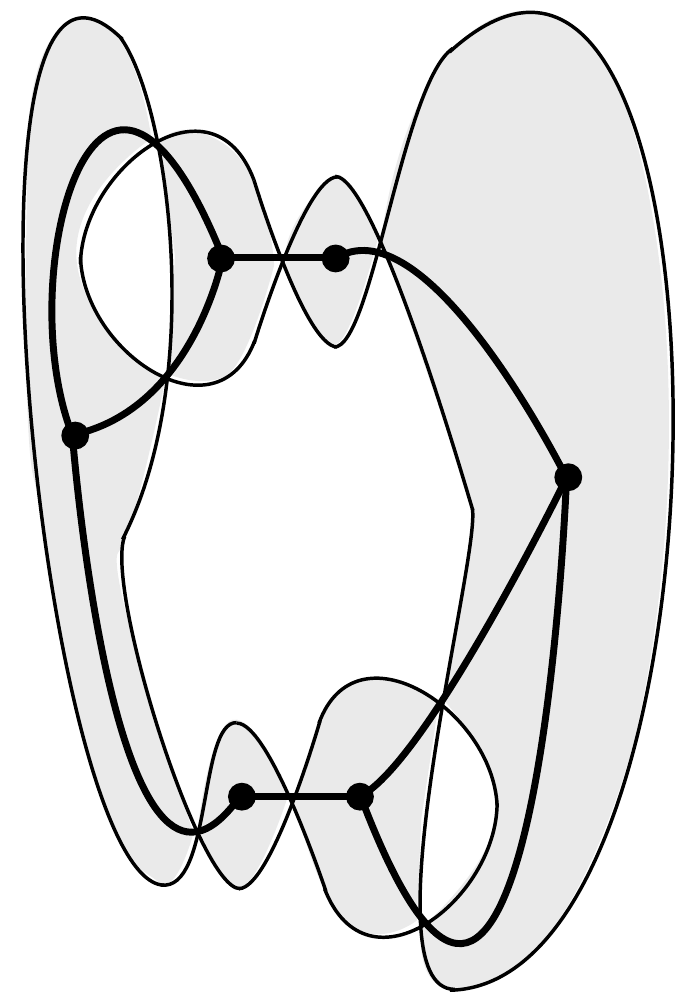}\end{center}&
 \Longstack{$\alpha_2$: color-,sign-preserving \\ $L$: alternating\\ $G$:   not self-dual \\ not antipodally self-dual \\ antipodally symmetric 
 }\\
\hline
\begin{center}\includegraphics[width=.9\linewidth]{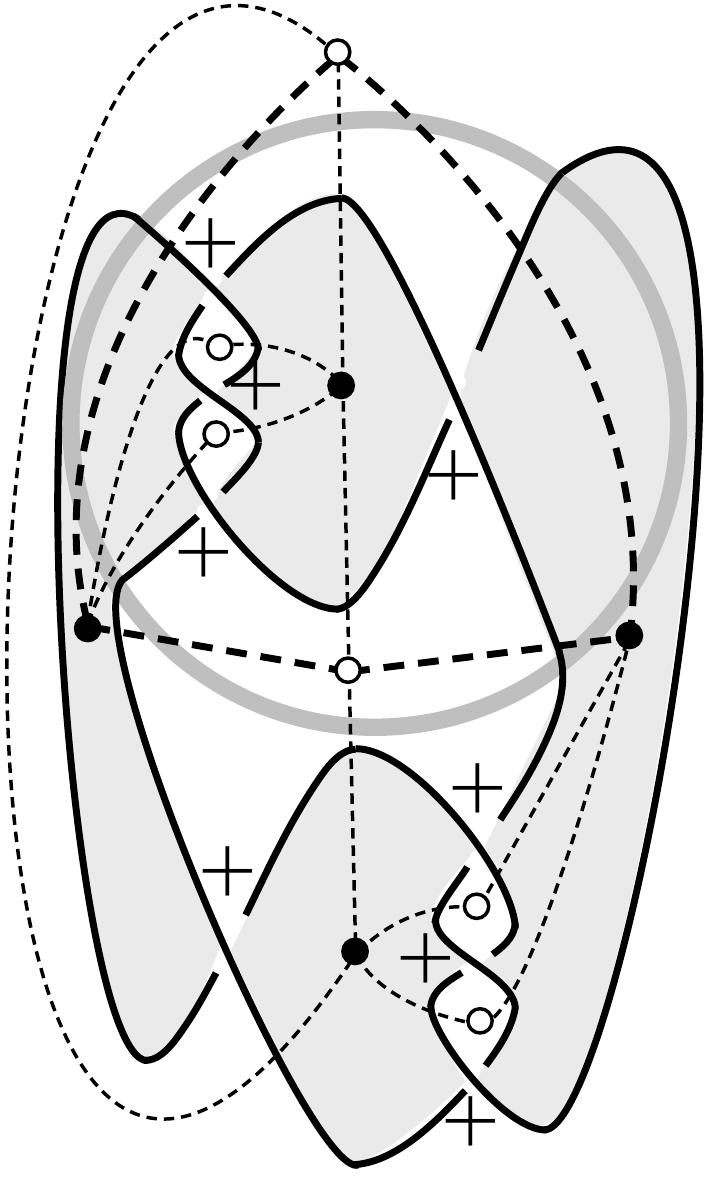}\end{center}&
\begin{center}\includegraphics[width=0.8\linewidth]{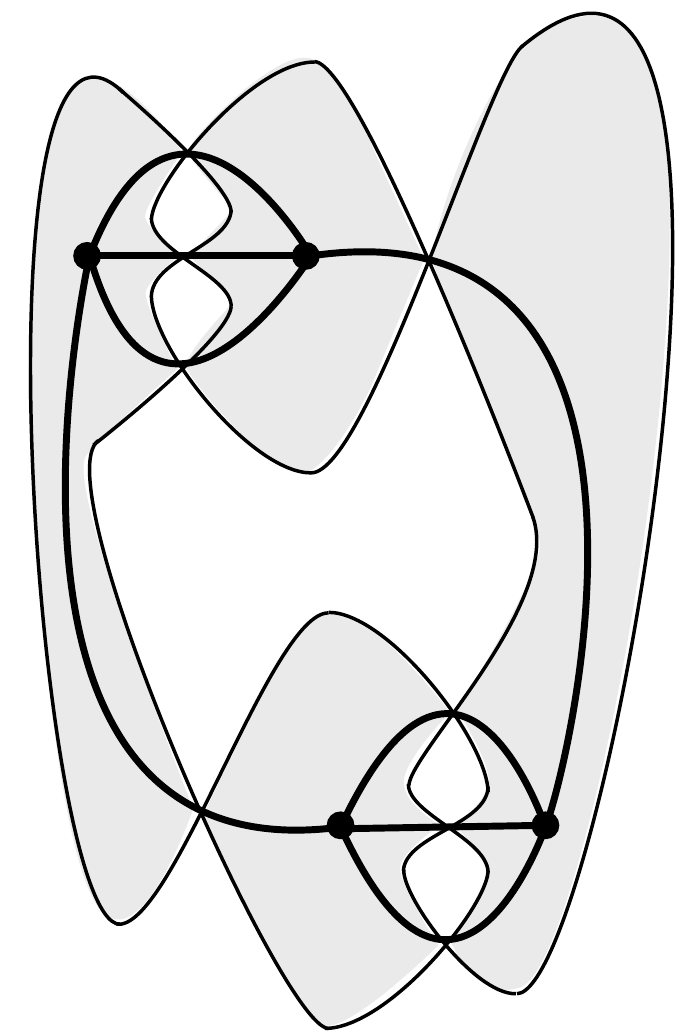}\end{center}&
 \Longstack{ $\alpha_2$: color-,sign-preserving \\  $L$: alternating\\ $G$:   not self-dual \\ not antipodally self-dual \\ antipodally symmetric 
 }\\
\hline
\begin{center}\includegraphics[width=1\linewidth]{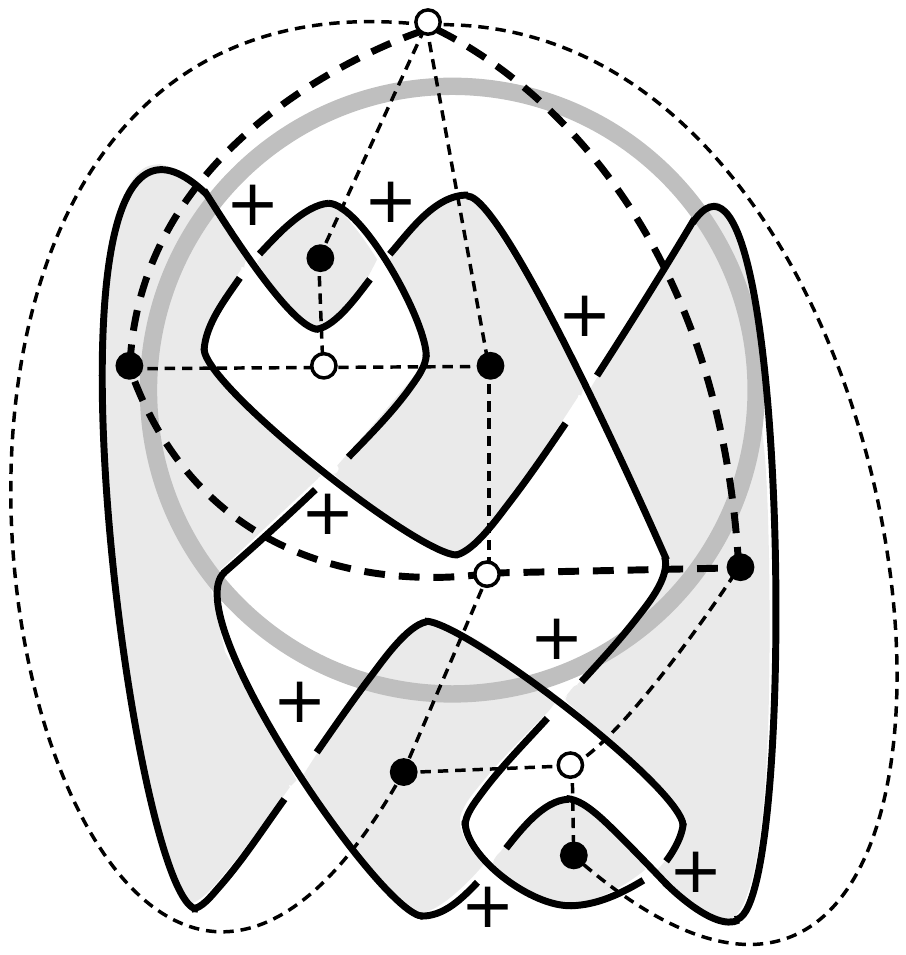}\end{center}&
\begin{center}\includegraphics[width=0.8\linewidth]{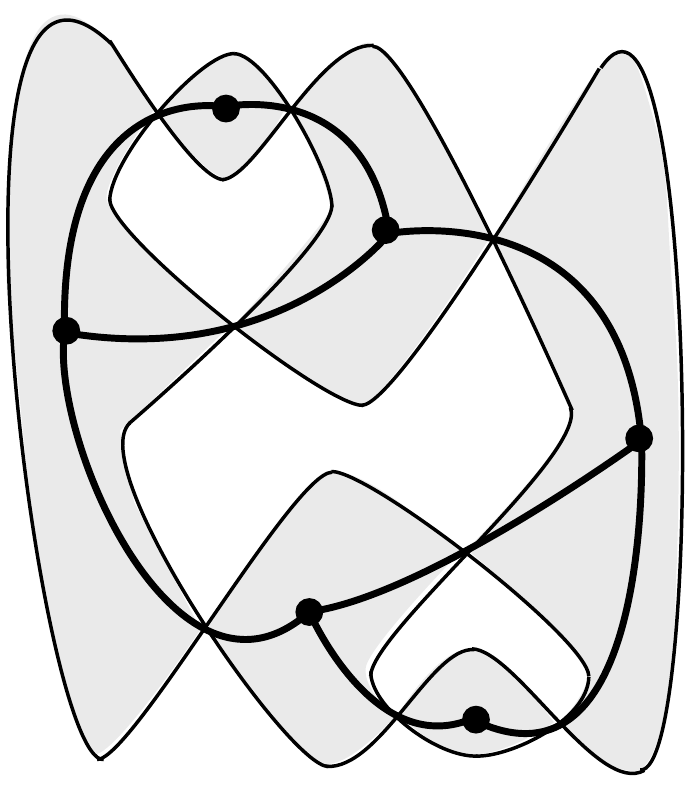}\end{center}&
 \Longstack{ $\alpha_2$: color-,sign-preserving \\  $L$: alternating\\ $G$:  not self-dual \\  not antipodally self-dual \\ antipodally symmetric 
 }\\
\hline
\end{tabular}
\end{table}


\begin{table}[ht]
\centering
\begin{tabular}{*{3}{m{0.32\textwidth}}}
\hline
\begin{center}\includegraphics[width=.9\linewidth]{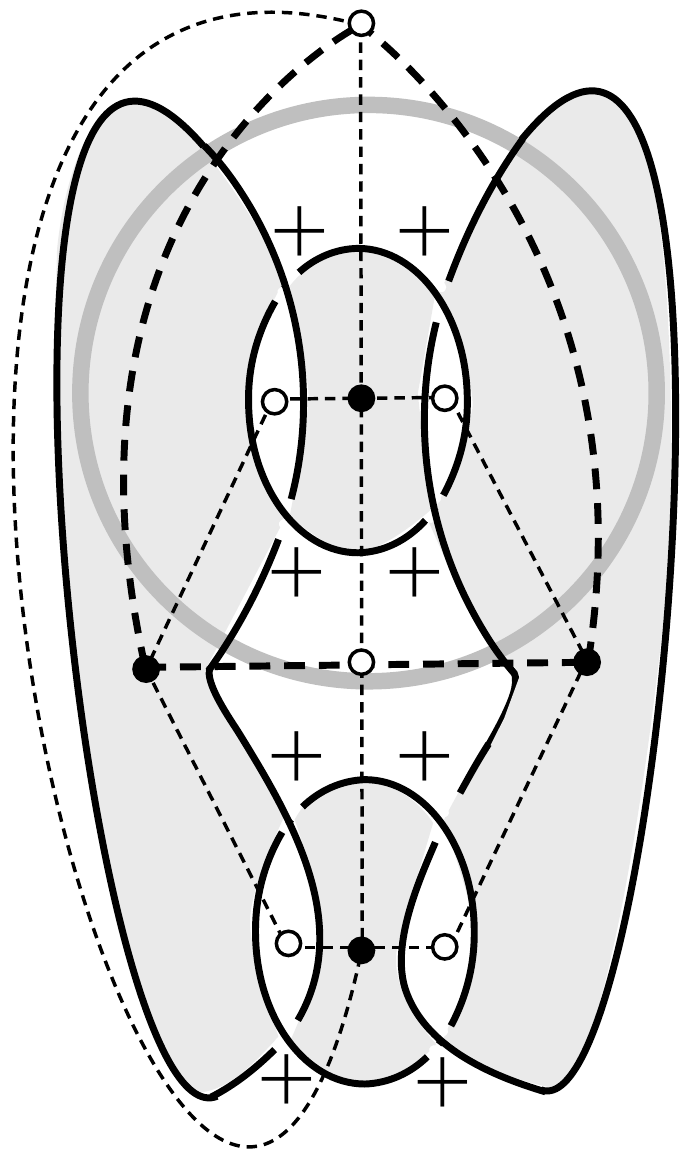}\end{center}&
\begin{center}\includegraphics[width=0.8\linewidth]{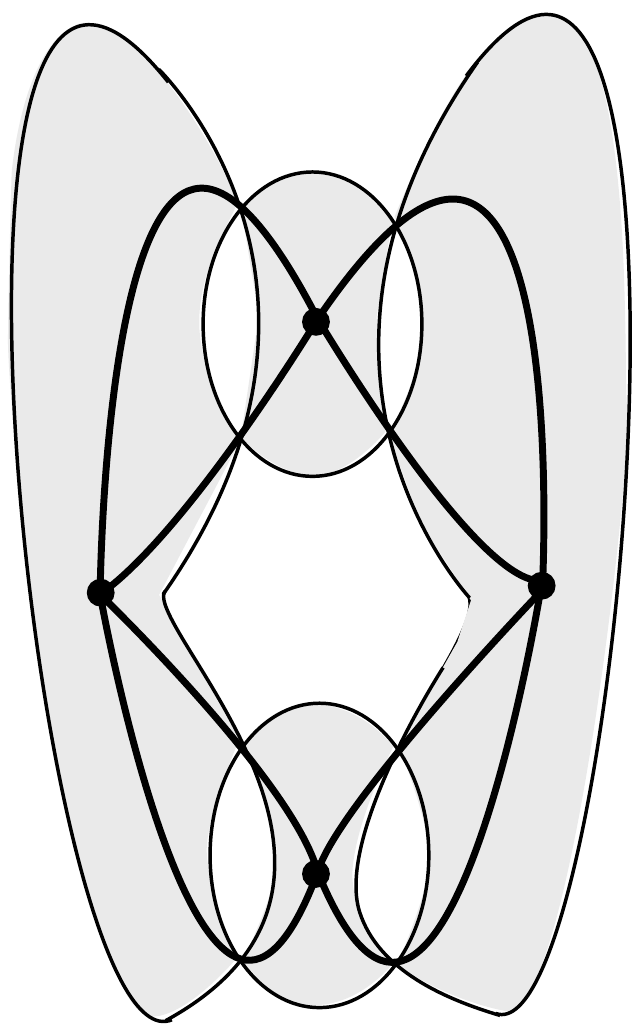}\end{center}&
 \Longstack{ $\alpha_2$: color-,sign-preserving \\ $L$: alternating\\  $G$:  not self-dual \\  not antipodally self-dual \\ antipodally symmetric 
 }\\
\hline
\begin{center}\includegraphics[width=.9\linewidth]{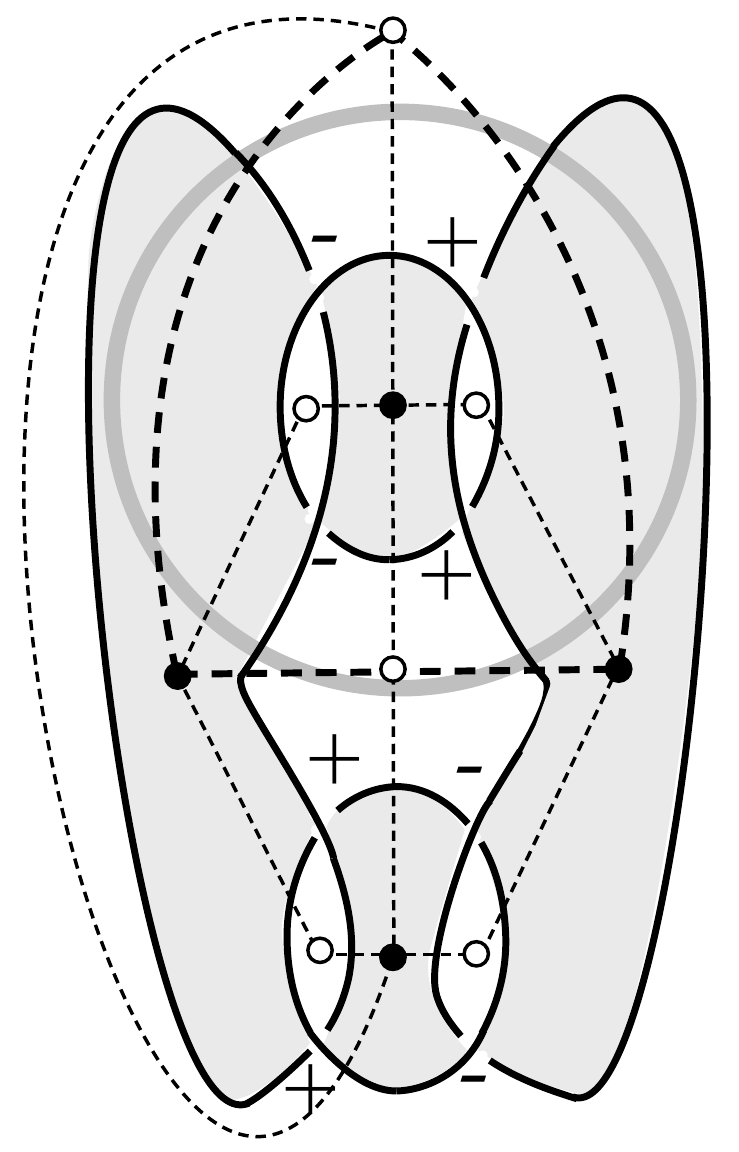}\end{center}&
\begin{center}\includegraphics[width=0.8\linewidth]{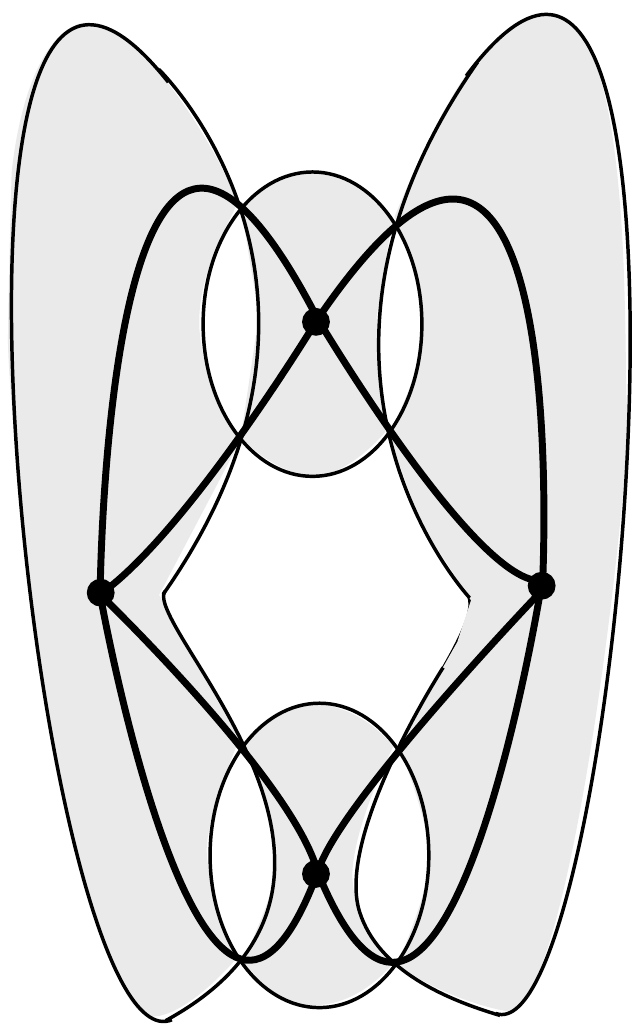}\end{center}&
 \Longstack{$\alpha_2$: color-,sign-preserving \\  $L$: nonalternating\\ $G$:  not self-dual \\  not antipodally self-dual \\ antipodally symmetric 
 }\\
\hline
\begin{center}\includegraphics[width=1.1\linewidth]{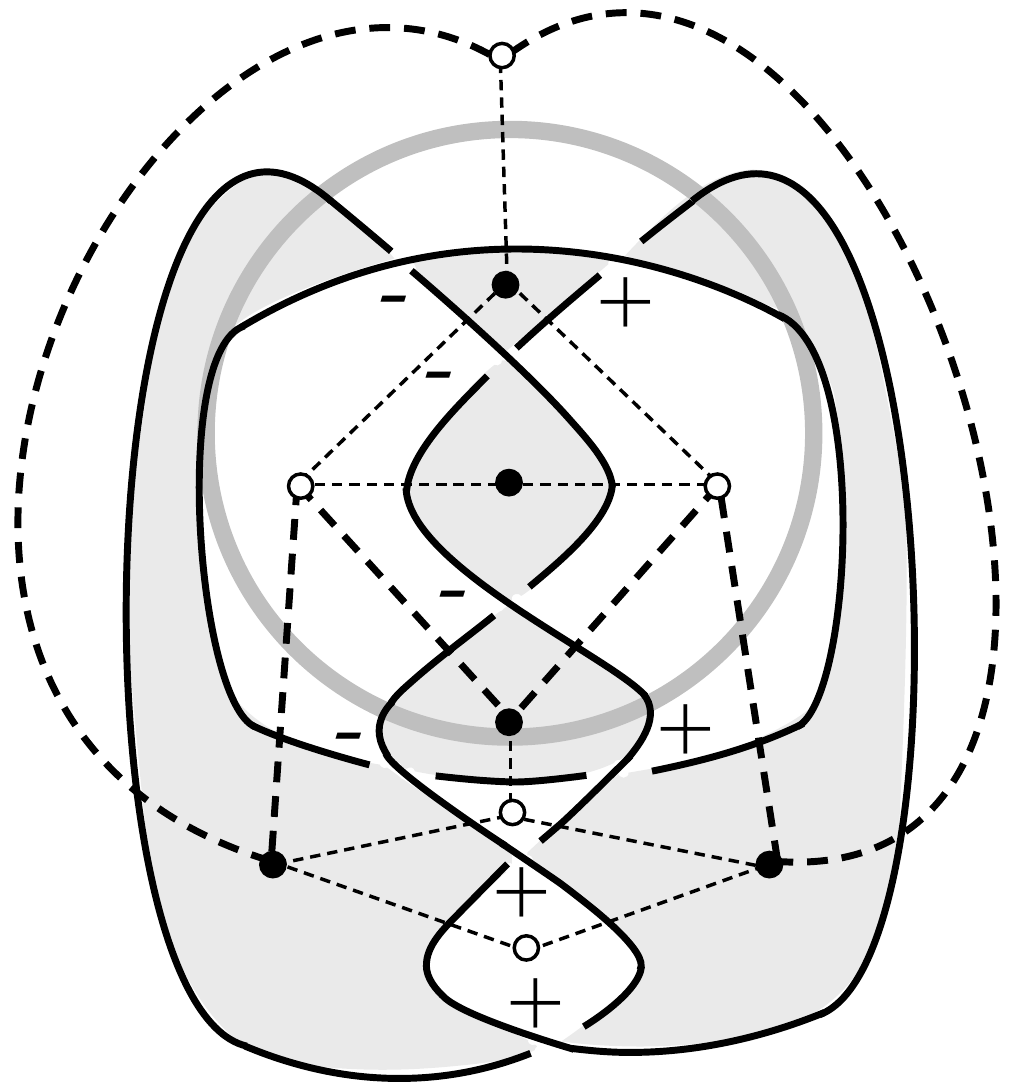}\end{center}&
\begin{center}\includegraphics[width=0.8\linewidth]{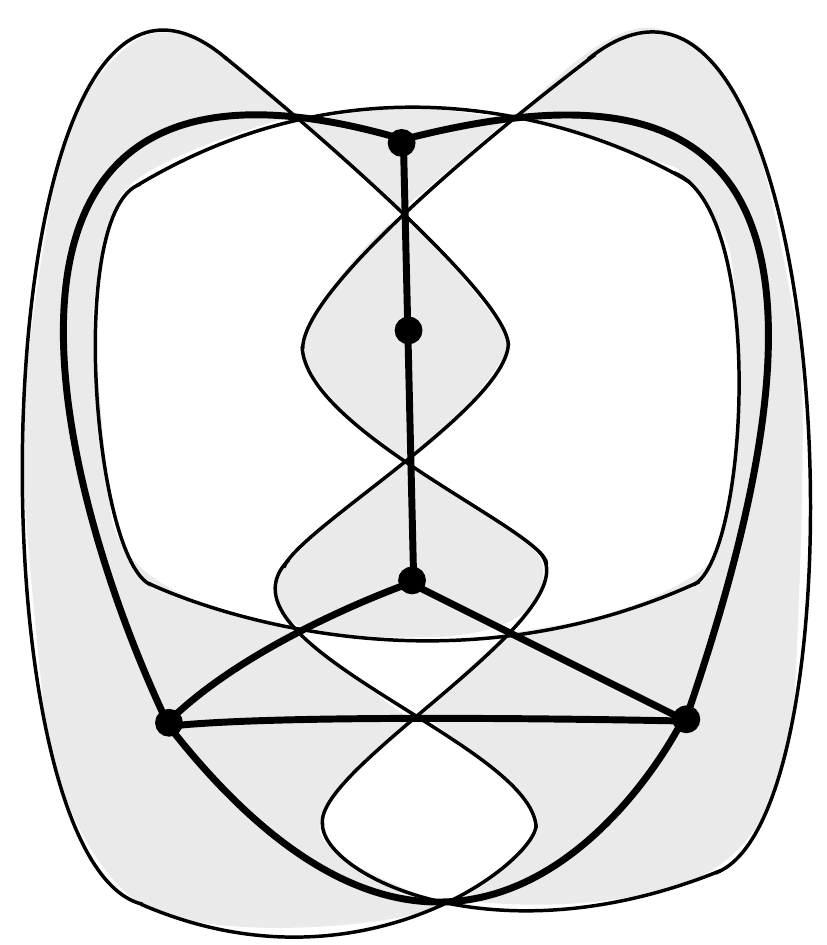}\end{center}&
 \Longstack{$\alpha_2$: color-,sign-reversing \\  $L$: nonalternating\\ $G$:   self-dual \\ antipodally self-dual \\ not antipodally symmetric 
 }\\
\hline
\end{tabular}
\end{table}

\begin{table}[ht]
\centering
\begin{tabular}{*{3}{m{0.32\textwidth}}}
\hline
\begin{center}\includegraphics[width=1.1\linewidth]{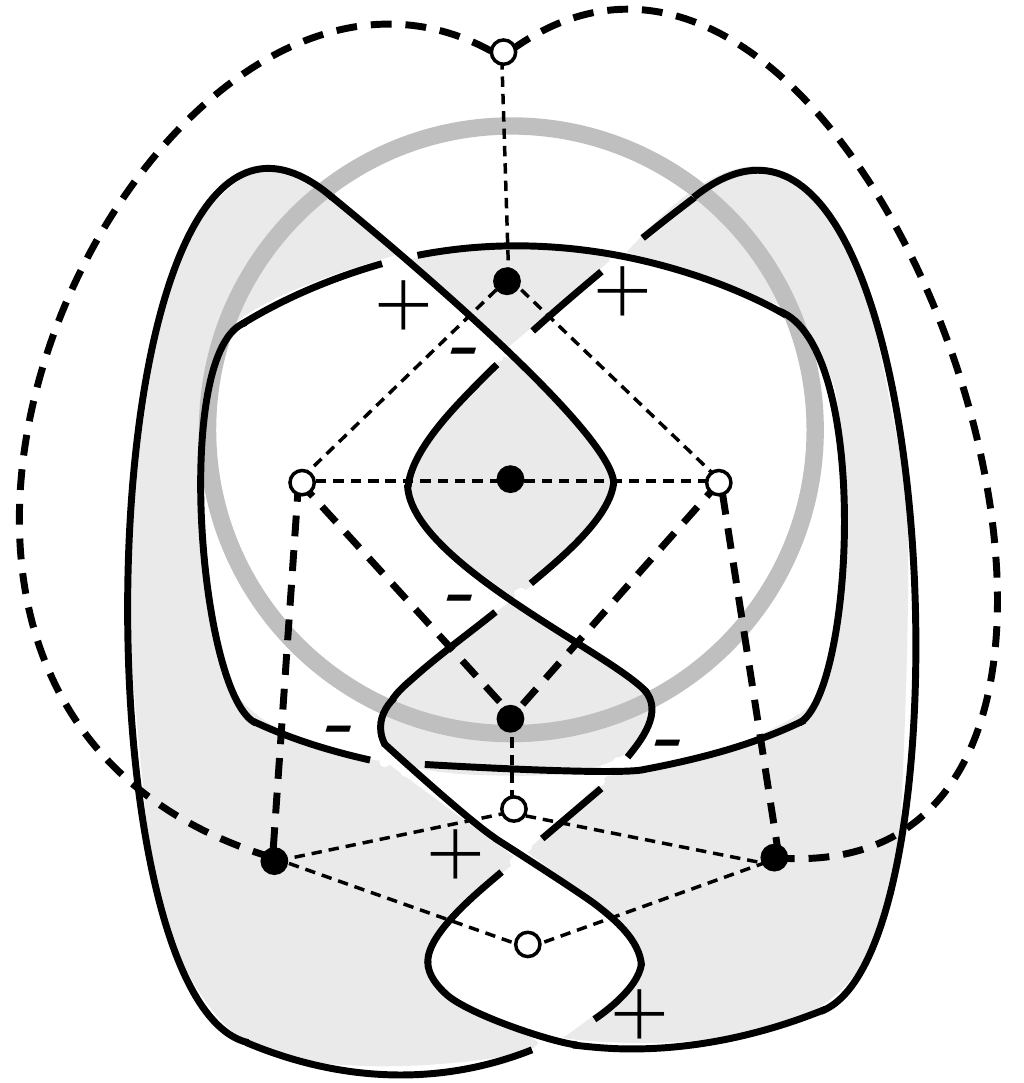}\end{center}&
\begin{center}\includegraphics[width=0.8\linewidth]{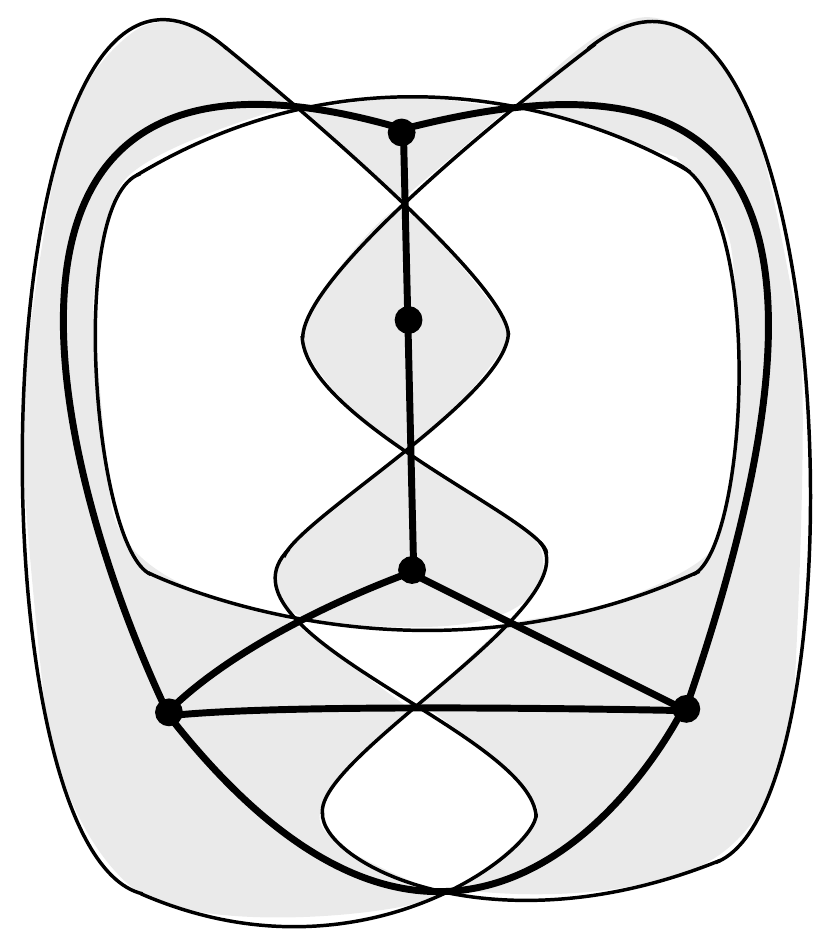}\end{center}&
 \Longstack{ $\alpha_2$: color-,sign-reversing \\ $L$: nonalternating\\ $G$:  self-dual \\ antipodally self-dual \\ not antipodally symmetric 
 }\\
\hline
\begin{center}\includegraphics[width=1.1\linewidth]{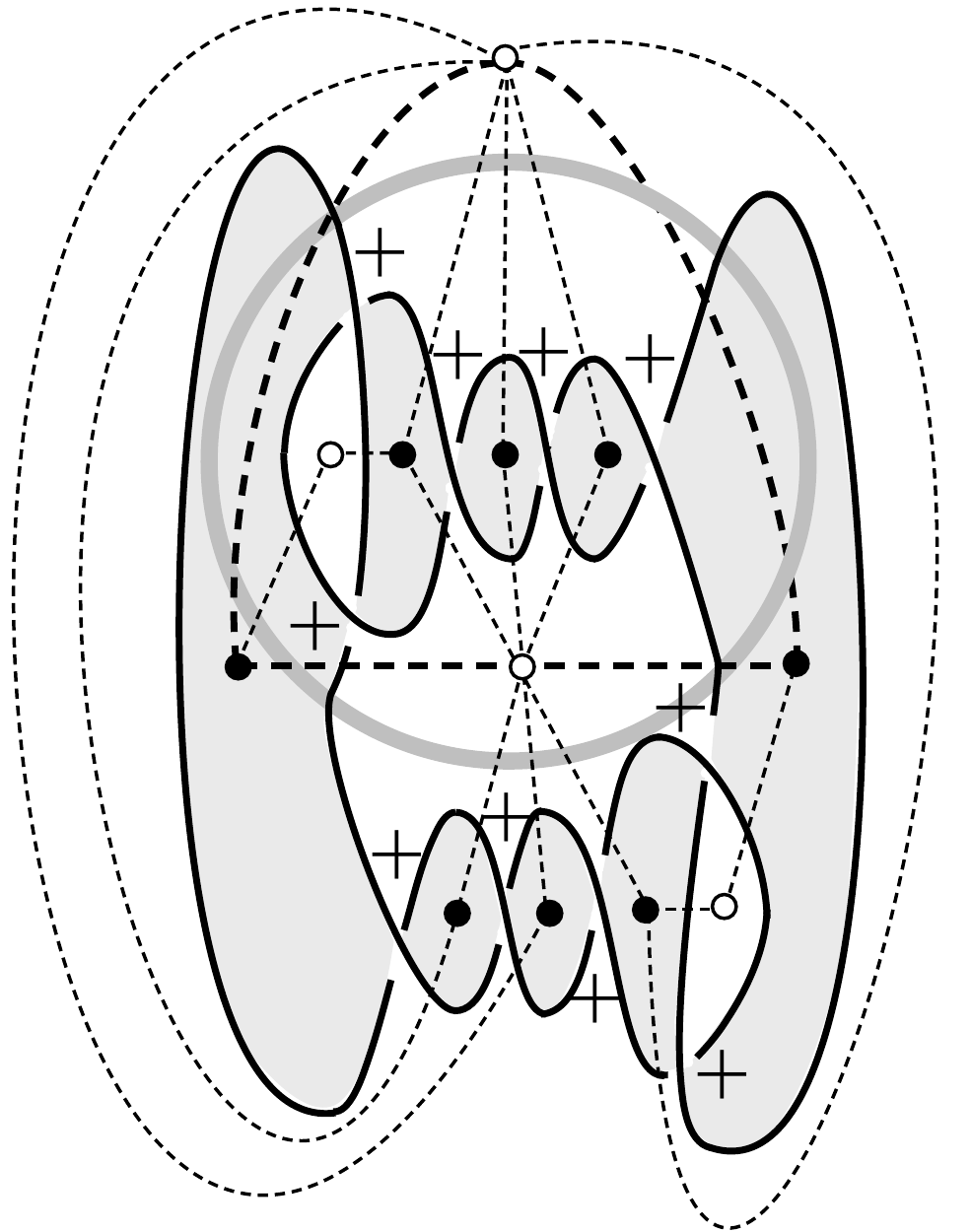}\end{center}&
\begin{center}\includegraphics[width=0.8\linewidth]{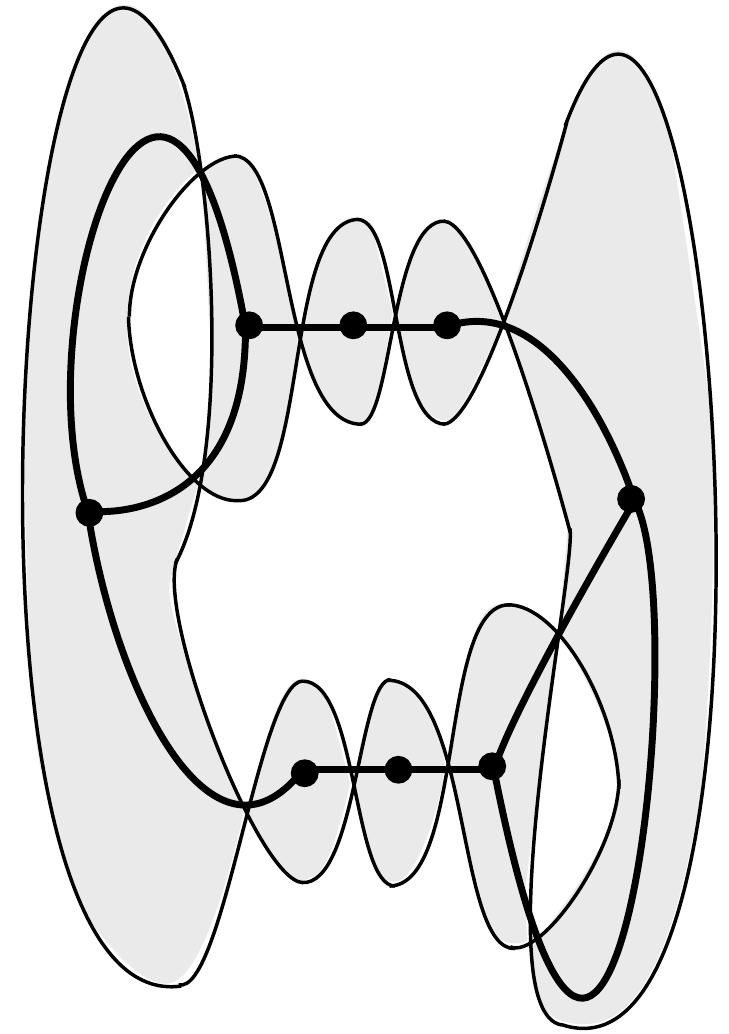}\end{center}&
 \Longstack{$\alpha_2$: color-,sign-preserving \\  $L$: alternating\\ $G$:   not self-dual \\  not antipodally self-dual \\ antipodally symmetric 
 }\\
\hline
\end{tabular}
\end{table}


\end{document}